\documentclass[11pt,oneside]{amsart}

\usepackage{amsmath,amssymb,amsthm}
\usepackage{mathrsfs}
\usepackage[english]{babel}
\usepackage[pagewise]{lineno}
\usepackage{geometry}
\usepackage[cal=boondoxo]{mathalfa}
\usepackage[colorlinks=true,urlcolor=blue,citecolor=red,linkcolor=blue]{hyperref}
\usepackage[hyperpageref]{backref}
\usepackage{fancyhdr}  % Pacote para personalizar cabeçalhos e rodapés

\usepackage[numbers, sort&compress]{natbib}
\numberwithin{equation}{section}
\newtheorem{theorem}{Theorem}[section]
\newtheorem{lemma}{Lemma}[section]

\newtheorem{remark}{Remark}[section]

\newcommand{\R}{\mathbb{R}}

\title{A Trudinger--Moser inequality under a refined constraint in fractional dimensions and extremal functions}
\author{Ruan Diego da Silva Paiva}
\author{José Francisco de Oliveira}\thanks{Second author was partially supported by  National Council for Scientific and Technological Development (CNPq) \# 309491/2021-5 and 303443/2025-1}
 \address[R.D. da Silva Paiva]{\newline\indent Department of Mathematics
	\newline\indent 
	State University of Piau\'{i}
	\newline\indent
	64600-000  Picos, PI, Brazil}
	\email{\href{mailto:ruanpaiva@pcs.uespi.br}{ruanpaiva@pcs.uespi.br}}
	
    \address[J.F.\ de Oliveira]{\newline\indent Department of Mathematics
	\newline\indent 
	Federal University of Piau\'{i}
	\newline\indent
	64049-550 Teresina, PI, Brazil}
	\email{\href{mailto:jfoliveira@ufpi.edu.br}{jfoliveira@ufpi.edu.br}}	
\begin{document}
\maketitle
\begin{abstract}
We establish a Trudinger–Moser type inequality with a Tintarev-type constraint in fractional-dimensional spaces and prove the existence of maximizers in the critical regime. Our results provide a refinement of those in (\textit{Calc. Var.} \textbf{52} (2015), 125–163) in the setting of fractional-dimensional spaces, as well as of those in (\textit{Ann. Global Anal. Geom.} \textbf{54} (2018), 237–256) for classical Sobolev spaces.
  \end{abstract}
 %  \tableofcontents
\section{Introduction}     
Let $\Omega$ be a smooth bounded domain in $\R^n\, (n \geq 2)$. Let $W^{1,n}_0(\Omega)$ be the limit case Sobolev space defined as the closure of $C_0^{\infty}(\Omega)$ with the Dirichlet norm   $\|u\|_D=(\int_\Omega |\nabla u|^n\mathrm{d} x)^{1/n}$. The classical Trudinger-Moser inequality asserts that
\begin{equation}\label{PTM1}
\sup_{u \in W^{1,n}_0(\Omega), \, \|\nabla u\|_{n}=
1}\int_{\Omega} e^{\mu |u|^{\frac{n}{n-1}}}\ \mathrm{d} x
 \quad
\left\{
  \begin{array}{lll}
    \leq C_n
|\Omega| & \mbox{if} & \mu\le\mu_{n}\\
    =\infty & \mbox{if} & \mu>\mu_{n},
  \end{array}
\right.
\end{equation}
where $\mu_n=n \omega_{n-1}^{1/(n-1)}$,  $|\Omega|$ denotes
the Lebesgue measure of a set $\Omega$ in $\mathbb{R}^n$ and
$\omega_{n-1}$ is the measure of the unit sphere in
$\mathbb{R}^n$. 
This estimate, established by J. Moser \cite{moser1971}, refines earlier contributions by Trudinger \cite{Trudinger67}, Yudovich \cite{Yudo}  and Pohozaev
\cite{Poho}, and has a vast number of applications and extensions in several settings, see for instance  \cite{Adams,AdDr2004,Tintarev2014,LamLu2,Tian} for some extensions  and \cite{Baird,Chang88,DjairoCPAM} for applications in geometric analysis and partial differential equations.

In the critical case $\mu=\mu_{n}$, the existence of extremal functions for \eqref{PTM1}  is a delicate problem, which was positively solved in a series of papers  \cite{carleson1986,Struwecritical,flucher1992,lin1996}.

Motivated by the refinements of \eqref{PTM1} established in \cite{Tintarev2014,Nguyen2017}, which also extend improvements in \cite{AdDr2004,Yang20006}, and by recent advances concerning Trudinger--Moser type estimates in fractional-dimensional settings \cite{JF2014,JFcasocritico,CV,DoLuPonciano2024a,DoLuPonciano2025,DOMADE,Ponciano,CCM02017,RMI}, in this paper we establish generalizations of \eqref{PTM1} to fractional-dimensional spaces and investigate the corresponding extremal problem.

To state our results precisely, let us briefly introduce the concept of fractional integrals and related weighted Sobolev spaces in which we will be working. From the formalism developed in \cite{Still77,Zubair12}, the integration of radially symmetric function $f(r)$ on a $\theta$-dimensional fractional space is given by
\begin{equation}\label{fractional integral}
    \int f(r(x_0,x_1)) \mathrm{d} x_0=\omega_{\theta}\int_{0}^{\infty}r^{\theta}f(r) \mathrm{d} r,
\end{equation}
where $r(x_0,x_1)$ is the distance between two points $x_0$ and $x_1$, and $\omega_{\theta}$ is defined  by
\begin{equation}\label{fractional volume}
\omega_{\theta}=\frac{2\pi^{\frac{\theta}{2}}}{\Gamma(\frac{\theta}{2})},\;\;\;\mbox{with}\;\;\;\Gamma(x)=\int_0^{\infty} t^{x-1} \operatorname{e}^{-t} \,
\mathrm{d} t,\;\; x>0.
\end{equation}
It is worth noting that integration over fractional dimensional spaces is often used in the dimensional regularization method as a powerful tool to obtain results in statistical mechanics and quantum field theory \cite{Collins, Palmer, Zubair11,Zubair12}. By simplicity we denote the corresponding $\theta$-fractional measure (cf. \eqref{fractional integral}) by
\begin{equation}\label{fractional-measure}
\begin{aligned}
\int_{0}^{R}f(r)\mathrm{d}\lambda_{\theta}=\omega_{\theta}\int_{0}^{R}f(r) r^{\theta}\mathrm{d}r,\quad \theta\ge0 \;\; \mbox{and}\;\; 0<R\le \infty.
\end{aligned}
\end{equation} 
On the other hand, based on Hardy-type inequalities in \cite{Opic}, Cl\'{e}ment-de Figueiredo-Mitidieri \cite{Clement-deFigueiredo-Mitidieri} proposed a class of weighted Sobolev spaces which are connected with fractional integrals and play important role in the study of a class of quasilinear elliptic equations including the $p$-Laplacian and $k$-Hessian operators in the radial form, see \cite{JF2014, JFcasocritico, CV, K-hessianattain}. Precisely, for  $0<R\le\infty$, $\theta\ge0$ and $q\ge 1$, set $L^q_{\theta}=L^q_{\theta}(0,R)$ the Lebesgue space associated with the $\theta$-fractional measure \eqref{fractional-measure} on the interval $(0,R)$ and let $AC_{loc}(0,R]$ be the set of of all locally absolutely continuous functions on the interval $(0,R]$.  Then, we denote by $X^{1,p}_R(\alpha,\theta)$ the weighted Sobolev spaces given by the completion of the set of all  functions  $u\in AC_{loc}(0,R)$ such that  $\lim_{r\rightarrow R}u(r)=0$, $u\in L^{p}_{\theta}$ and $u^{\prime}\in L^{p}_{\alpha}$ with the norm
\begin{equation}\label{Xnorm-full}
\|u\|=(\|u\|^{p}_{L^{p}_{\theta}}+\|u^{\prime}\|^{p}_{L^{p}_{\alpha}})^{\frac{1}{p}}.
\end{equation}
We recall that  \eqref{Xnorm-full}  is equivalent to  the gradient norm $\|u^{\prime}\|_{L^{p}_{\alpha}}$ under the condition
\begin{equation}\label{cond-equivNorm}
\theta\ge \alpha-p \;\;\mbox{and}\;\;  0<R<\infty,
\end{equation}
see for instance \cite{JFcasocritico}.  In fact, the behavior of functions in $X^{1,p}_R(\alpha,\theta)$ depends on the parameters $\alpha$, $p$, and $\theta$, leading to three regimes: the Sobolev case for $\alpha-p+1>0$, the Trudinger–Moser case for $\alpha-p+1=0$, and the Morrey case for $\alpha-p+1<0$, see for instance \cite{DOMADE,Ponciano}.  If  $\alpha-p+1>0$ and $0<R<\infty$, one has the continuous embedding 
\begin{equation}\label{imersão-Sobolev}
	X^{1,p}_R(\alpha,\theta)\hookrightarrow L^q_{\theta}\quad\mbox{for all}\quad 1<q\le p^*\;\;\mbox{and}\;\; \theta\ge \alpha-p
\end{equation}
where the critical exponent $p^*$ is given by
\begin{equation}
	 p^*=\frac{(\theta+1)p}{\alpha-p+1}.
\end{equation}
The embedding \eqref{imersão-Sobolev}  is compact in the strict case $\theta>\alpha-p$ and $q<p^*$. In this work, we are interested in the Trudinger–Moser case on bounded domains. Therefore, from now on, unless otherwise stated,  we shall assume
\begin{equation}\label{TM-case}
  \alpha-p+1=0\;\;\mbox{and}\;\; 0<R<\infty.
\end{equation}
Note that, under such conditions,  \eqref{Xnorm-full} is also equivalent to the gradient norm $\|u^{\prime}\|_{L^{p}_{\alpha}}$, since $\theta\ge \alpha-p=-1$. In addition, we have the  compact embedding 
\begin{equation}\label{imersão}
 X^{1,p}_R(\alpha,\theta)\hookrightarrow L^q_{\theta}\quad\mbox{for all}\quad q\in (1,\infty).
\end{equation}
The threshold growth for the embedding \eqref{imersão} is not attained for any Lebesgue space $L^q_{\theta}$ but it is given by the Orlicz space determinated by  the exponential function, as demonstrated in \cite{JF2014,CCM02017}. In fact, in  \cite{JF2014} is proved that $\exp(\mu|u|^{{\frac{p}{p-1}}})\in L^1_{\theta}$, for any $u\in X^{1,p}_R(\alpha,\theta)$ and $\mu>0$.  In addition, there exists $c< \infty$, depending on $\alpha, \theta, p$ and $R$ such that
\begin{equation}\label{TM1}
\sup_{u \in X^{1,p}_R(\alpha,\theta), \; \|u^{\prime}\|_{L^p_{\alpha}}\le 1} \int_0^R e^{\mu |u|^\frac{p}{p-1}}\,\mathrm{d}\lambda_{\theta}\; \begin{cases}
 \leq c, &\mbox{if }
 \;\mu \leq \mu_{\alpha,\theta}\\
 =\infty, &\mbox{if }\; \mu>\mu_{\alpha, \theta}
\end{cases}
\end{equation}
with $\mu_{\alpha,\theta}=(\theta+1)\omega_\alpha^\frac{1}{\alpha}$. Further, \eqref{TM1} admits an extremal function.
In \cite{JFcasocritico}, the authors established an extension of Adimurthi-Druet type \cite{AdDr2004,Yang2006} for \eqref{TM1}. More precisely,
\begin{equation}\label{PTM5}
  \sup_{u \in X^{1,p}_R(\alpha,\theta), \; \|u^{\prime}\|_{L^p_{\alpha}}\le 1} \int_0^R e^{\mu_{\alpha,\theta}(1+\nu\|u\|_{L^{p}_\theta}^{p})^{\frac{1}{p-1}} |u|^\frac{p}{p-1}}\,\mathrm{d}\lambda_{\theta}<\infty
\end{equation}
for any $\nu>0$ such that
\begin{equation*}
		0\leq \nu <\lambda_{\alpha,\theta}=\inf_{X^{1,p}_R(\alpha,\theta) \backslash \{0\}}\frac{\|u^\prime\|^{p}_{L^{p}_\alpha}}{\|u\|^{p}_{L^{p}_\theta}}.
	\end{equation*}
Furthermore, the supremum in \eqref{PTM5} is attained for some $v \in X^{1,p}_R(\alpha,\theta)$. For further extensions of \eqref{TM1}, we refer to \cite{CV,RMI,Ponciano,DoLuPonciano2024a,DoLuPonciano2025,DOMADE,CCM02017} and the references therein. 

Our aim here is to establish an improvement of the Trudinger–Moser inequality  \eqref{PTM5} and provide the fractional-dimensional counterpart of \cite{Tintarev2014,Nguyen2017}. First, we fix the notation
\begin{equation}\label{H}
    H_{\nu}(u)=(\|u^\prime\|^{p}_{L^{p}_\alpha}-\nu\|u\|^{p}_{L^{p}_\theta})^{\frac{1}{p}}, \;\; \mbox{with }\;\; 0\leq\nu<\lambda_{\alpha,\theta}.
\end{equation}
Our first result reads as follows:
\begin{theorem}\label{teoprinc1}
    Let $p \geq 2$, $\alpha$  and $R$ satisfy assumption \eqref{TM-case}, and let $\theta\ge \alpha$. Then 
    \begin{equation}\label{CT}
		S(p,\nu,\theta,R) = \sup_{u \in X^{1,p}_R(\alpha,\theta), \; H_{\nu}(u) \leq 1} \int_0^R e^{\mu_{\alpha,\theta}|u|^{\frac{p}{p-1}}} \mathrm{d}\lambda_{\theta}<\infty
	\end{equation}
    for any $0\leq \nu<\lambda_{\alpha,\theta}$.
\end{theorem}
Note that the estimate \eqref{CT} is stronger than  \eqref{PTM5}. Indeed, let $u \in X^{1,p}_R(\alpha,\theta)$ with  $\|u^{\prime}\|^{p}_{L^{p}_\alpha} \leq 1$ and set $v=(1+\nu\|u\|^{p}_{L^{p}_\theta})^\frac{1}{p}u$. Since $0 \leq \nu < \lambda_{\alpha,\theta}$ we have 
	\begin{align*}
	  H_{\nu}^{p}(v)
        = \|u^{\prime}\|^{p}_{L^{p}_\alpha}+\nu\|u\|^{p}_{L^{p}_\theta}\big(\|u^{\prime}\|^{p}_{L^{p}_\alpha}-1\big)-\nu^2\|u\|^{2p}_{L^{p}_\theta}\leq \|u^{\prime}\|^{p}_{L^{p}_\alpha}\le 1.  
	\end{align*}
Thus, by applying the estimate \eqref{CT} to the function $v$, we obtain \eqref{PTM5}.
	\begin{theorem}\label{teoprinc2} Suppose the assumptions of Theorem~\ref{teoprinc1} hold. Then, $S(p,\nu,\theta,R)$ is attained for some function
	 $u_0 \in X^{1,p}_R(\alpha,\theta) \cap C^{1}[0,R]$ such that $H_{\nu}(u_0)=1$.
	\end{theorem}
    Theorems \ref{teoprinc1} and \ref{teoprinc2} provide an alternative in fractional-dimensional spaces to the recent result obtained by Nguyen \cite{Nguyen2017} in the context of classical Sobolev spaces $W^{1,n}_0(\Omega)$. In the same spirit as \cite{JFcasocritico,CV}, to prove both Theorem~\ref{teoprinc1} and Theorem~\ref{teoprinc2} we combine a Lions-type uniform estimate  with blow-up analysis and refined test-function computations. To make the blow-up procedure available in the fractional-dimensional setting, classification results, the construction of Green-type functions, and delicate computations involving special functions are required.
    
	\section{The subcritical inequalities and their extremal functions}
    In this section we prove both Theorem~\ref{teoprinc1} and Theorem~\ref{teoprinc2} for the subcritical regime. Precisely,  let $\mu_\varepsilon = \mu_{\alpha,\theta} - \varepsilon$ with $0<\varepsilon<\mu_{\alpha,\theta}$ and define
	\begin{equation}\label{caeceps}
		S_{\varepsilon}(p,\nu,\theta,R) = \sup_{u \in X^{1,p}_R(\alpha,\theta),\; H_{\nu}(u) \leq 1} \int_0^R e^{\mu_\varepsilon|u|^{\frac{p}{p-1}}} \mathrm{d}\lambda_{\theta}, 
	\end{equation}
    where $0 \leq \nu < \lambda_{\alpha,\theta}$. We will show that $S_{\varepsilon}(p,\nu,\theta,R)<\infty$ and it is attained. The first ingredient in our proof is the following Lions-type estimate.
	\begin{lemma}[Lions-type estimate]\label{improvement lions}
	 Suppose $0 \leq \nu < \lambda_{\alpha,\theta}$. Let $(u_j) \in X^{1,p}_R(\alpha,\theta)$ be a sequence such that $H_{\nu}(u_j) = 1$ and $u_j \rightharpoonup u$ in $X^{1,p}_R(\alpha,\theta)$. Then, for any $0 < q < P_{H}(u)$, we have
		\begin{equation}\label{e-Lions}
			\limsup_{j \to \infty} \int_0^R e^{q\mu_{\alpha,\theta}|u_j|^\frac{p}{p-1}} \mathrm{d}\lambda_{\theta} < \infty,
		\end{equation}
        where
        \begin{equation}
           P_H(u)=\left\{\begin{aligned}
            & (1 - H^p_{\nu}(u))^{-\frac{1}{p-1}},\;\;&\mbox{if}&\;\;    H_{\nu}(u)<1\\
            &+\infty\;\; &\mbox{if}&\;\;  H_{\nu}(u)=1.
            \end{aligned}\right.
        \end{equation}
	\end{lemma}
	
	\begin{proof} If  $u \equiv 0$,  then $H_{\nu}(u)=0$ and \eqref{imersão} ensures $\|u_j^\prime\|^{p}_{L^{p}_\alpha}=1+\nu\|u_j\|^{p}_{L^{p}_\theta}\to 1$ as $j \to \infty$. Thus, \eqref{TM1} yields \eqref{e-Lions} for $0<q<1$. Assume $u\not\equiv 0$. Hence,
    \begin{equation}\label{grad-no0)}
        \|u_j^\prime\|^{p}_{L^{p}_\alpha} = 1 + \nu\|u_j\|^{p}_{L^{p}_\theta} \to 1 + \nu\|u\|^{p}_{L^{p}_\theta},\;\;\mbox{as}\;\;j \to \infty.
    \end{equation}
    Set $v_j = u_j \slash \|u_j^\prime\|_{L^{p}_\alpha}$. Then 
    \begin{equation}\label{u_jtov}
     \|v_j^\prime\|_{L^{p}_\alpha} = 1\;\;\mbox{and} \;\; v_j \rightharpoonup v := u/(1 + \nu\|u\|^{p}_{L^{p}_\theta})^{1/p}\;\;\mbox{in}\;\;  X^{1,p}_R(\alpha,\theta).    
    \end{equation}
    Suppose $ H_{\nu}(u)<1$. Then, for $q < (1 - H^p_{\nu}(u))^{-\frac{1}{p-1}}$, from \eqref{grad-no0)} we have 
		\begin{eqnarray*}
			\lim_{j \to \infty} q\|u^\prime_j\|^{\frac{p}{p-1}}_{L^{p}_\alpha} &=& q(1 + \nu \|u\|^{p}_{L_\theta^{p}})^\frac{1}{p-1}\\
			&<& \left(\frac{1 + \nu\|u\|^{p}_{L_\theta^{p}}}{1 - \|u^\prime\|^{p}_{L_\alpha^{p}} + \nu\|u\|^{p}_{L_\theta^{p}}}\right)^\frac{1}{p-1} = (1 - \|v^{\prime}\|^{p}_{L_\alpha^{p}})^{-\frac{1}{p-1}}.
		\end{eqnarray*}
	Consequently, we can choose $\overline{q} < (1 - \|v^{\prime}\|^{p}_{L_\alpha^{p}})^{-\frac{1}{p-1}}$ such that $q\|u^\prime_j\|^{\frac{p}{p-1}}_{L^{p}_\alpha} < \overline{q}$, for $j$ sufficiently large. Thus,  \cite[Theorem~1]{JFcasocritico} ensures
		\begin{equation*}
        \begin{aligned}
        \limsup_{j \to \infty}\int_0^R e^{\mu_{\alpha,\theta}q|u_j|^\frac{p}{p-1}} \mathrm{d}\lambda_{\theta} &=\limsup_{j \to \infty} \int_0^R e^{\mu_{\alpha,\theta}q\|u^\prime_j\|_{L^{p}_\alpha}^\frac{p}{p-1}|v_j|^\frac{p}{p-1}} \mathrm{d}\lambda_{\theta}\\
        &\le \limsup_{j \to \infty} \int_0^R e^{\mu_{\alpha,\theta}\overline{q}|v_j|^\frac{p}{p-1}} \mathrm{d}\lambda_{\theta}<\infty,
        \end{aligned}
		\end{equation*}
        for any $0 < q <(1 - H^p_{\nu}(u))^{-\frac{1}{p-1}}$.

        If $ H_{\nu}(u)=1$, then $\|v^{\prime}\|^{p}_{L^{p}_{\alpha}}=1$. Since $(1-\frac{\nu}{\lambda_{\alpha,\theta}})\|u^{\prime}_j\|^p_{L^p_{\alpha}}\le H^p_{\nu}(u_j)=1$ ensures $\|u^{\prime}_j\|^{p/(p-1)}_{L^p_{\alpha}}\le \bar{c}$, from \cite[Theorem~1]{JFcasocritico} again we have
        \begin{equation*}
        \begin{aligned}
        \limsup_{j \to \infty}\int_0^R e^{\mu_{\alpha,\theta}q|u_j|^\frac{p}{p-1}} \mathrm{d}\lambda_{\theta} &=\limsup_{j \to \infty} \int_0^R e^{\mu_{\alpha,\theta}q\|u^\prime_j\|_{L^{p}_\alpha}^\frac{p}{p-1}|v_j|^\frac{p}{p-1}} \mathrm{d}\lambda_{\theta}\\
        &\le \limsup_{j \to \infty} \int_0^R e^{\mu_{\alpha,\theta}q\overline{c}|v_j|^\frac{p}{p-1}} \mathrm{d}\lambda_{\theta}<\infty,
        \end{aligned}
		\end{equation*}
        for any $0<q<P_{H}(u)=+\infty$.
	\end{proof}
	Now, we are in  a position to prove our results in the subcritical regime. 
	\begin{lemma}\label{prop1}
		Suppose $p\ge 2 $ and $\theta \geq \alpha =p-1$. For each $\varepsilon > 0$ we have $S_{\varepsilon}(p,\nu,\theta,R)<\infty$. Moreover, there exists a  function $u_\varepsilon \in X^{1,p}_R(\alpha,\theta) \cap C^1[0,R]$ such that 
		\begin{equation*}
			S_{\varepsilon}(p,\nu,\theta,R)  = \int_0^R e^{\mu_\varepsilon|u_\varepsilon|^{\frac{p}{p-1}}} \mathrm{d}\lambda_{\theta},\;\;\; H_{\nu}(u_\varepsilon) = 1
		\end{equation*}
		and
		\begin{equation}\label{Lagrange}
			\left\{\begin{aligned}
			 & \int_0^R |u_\varepsilon^\prime|^{p-2} u_\varepsilon^\prime v^\prime \mathrm{d}\lambda_\alpha = \frac{1}{\lambda_\varepsilon} \int_0^Re^{\mu_\varepsilon u_\varepsilon^\frac{p}{p-1}} u_\varepsilon^\frac{1}{p-1}v\mathrm{d}\lambda_{\theta} + \nu \int_{0}^{R}u_\varepsilon^{p-1}v \mathrm{d}\lambda_{\theta},\;\; \forall\; v\in X^{1,p}_R(\alpha,\theta)\\
				&\lambda_\varepsilon = \int_0^R e^{\mu_\varepsilon u_\varepsilon^\frac{p}{p-1}} u_\varepsilon^\frac{p}{p-1} \mathrm{d}\lambda_{\theta} \\
				&u_\varepsilon \in X^{1,p}_R(\alpha,\theta), \;\; u_\varepsilon \geq 0 \;\;\mbox{and}\;\; H_{\nu}(u_\varepsilon) = 1.  
			\end{aligned}\right.
		\end{equation}
	 Furthermore,
		\begin{equation*}
			\lim_{\varepsilon \to 0} S_{\varepsilon}(p,\nu,\theta,R) =S(p,\nu,\theta,R)
		\end{equation*}
    and 
    \begin{equation}\label{lambda-from below}
        \liminf_{\varepsilon \to 0} \lambda_\varepsilon>0.
    \end{equation}
	\end{lemma}
	
	\begin{proof}
		Let $(u_j) \subset X^{1,p}_R(\alpha,\theta)$ be a maximizing sequence for $S_{\varepsilon}(p,\nu,\theta,R)$, that is, 
		\begin{equation}
		H_{\nu}(u_j) \leq 1\;\;\mbox{and}\;\;	\lim_{j \to \infty} \int_0^R e^{\mu_\varepsilon |u_j|^{\frac{p}{p-1}}} \mathrm{d}\lambda_{\theta} = S_{\varepsilon}(p,\nu,\theta,R).
		\end{equation}
		Note that $(|u_j|)^\prime=\operatorname{sgn}(u_j)u^\prime_j$. Therefore $H_{\nu}(|u_j|) \leq H_{\nu}(u_j) \leq 1$, so $|u_j|$ is also a maximizing sequence for $S_{\varepsilon}(p,\nu,\theta,R)$. Thus, we can replace $u_j$ with $|u_j|$ and assume that $u_j$ is non-negative. From the definition of $\lambda_{\alpha,\theta}$, we have
		\begin{equation*}
		1 \geq \|u_j^\prime\|^{p}_{L^{p}_\alpha} - \nu\|u_j\|^{p}_{L^{p}_\theta} \geq \left(1 - \frac{\nu}{\lambda_{\alpha,\theta}}\right)\|u_j^\prime\|^{p}_{L^{p}_\alpha}.
		\end{equation*}
		Since $\nu < \lambda_{\alpha,\theta}$, the sequence $(u_j)$ is bounded in $X^{1,p}_R(\alpha,\theta)$. Therefore, \eqref{imersão} yields
		\begin{equation*}
			u_j \rightharpoonup u_\varepsilon \text{ in } X^{1,p}_R(\alpha,\theta), \quad u_j \to u_\varepsilon \quad\text{ in }\quad L^{p}_\theta, \quad u_j(r) \to u_\varepsilon(r)\quad \text{a.e. in}\quad (0,R).
		\end{equation*}
	By weak convergence, we obtain $\|u_\varepsilon^\prime\|^{p}_{L^{p}_\alpha} \leq \liminf \|u_j^\prime\|^{p}_{L^{p}_\alpha}$, hence $H_{\nu}(u_\varepsilon) \leq 1$. Naturally, $u_\varepsilon \geq 0$. We claim that $ u_\varepsilon \not\equiv 0 $. Indeed, if $ u_\varepsilon \equiv 0 $, we have 
		$$
		\limsup \|u_j^\prime\|^{p}_{L^{p}_\alpha} \leq 1 + \nu\limsup \|u_j\|^{p}_{L^{p}_\theta} \leq 1.
		$$
		By the Trudinger-Moser inequality in \cite{JF2014}, the sequence $ e^{\mu_\varepsilon |u_j|^{\frac{p}{p-1}}} $ is uniformly bounded in $ L^q_\theta $ for $ 1 < q < \mu_{\alpha,\theta}/{\mu_\varepsilon} $. Using Vitali’s convergence theorem, we have
		
		\begin{equation}
			S_{\varepsilon}(p,\nu,\theta,R)=\lim_{j\to \infty}\int_0^Re^{\mu_\varepsilon|u_j|^\frac{p}{p-1}}\mathrm{d}\lambda_{\theta}= \int_0^R \mathrm{d}\lambda_{\theta}=|B_R|_{\theta}
		\end{equation}
	which is impossible, as we will see next. Thus, we conclude the claim. Now, if $u_\varepsilon \not\equiv 0$, by Lemma~\ref{improvement lions}, the sequence $ e^{\mu_\varepsilon |u_j|^\frac{p}{p-1}} $ is bounded in $ L^q_\theta $ for some $ q > 1 $. Consequently, we have 
		\begin{equation*}
			S_{\varepsilon}(p,\nu,\theta,R) = \lim_{j \to \infty} \int_0^R e^{\mu_\varepsilon |u_j|^\frac{p}{p-1}} \mathrm{d}\lambda_{\theta} = \int_0^R e^{\mu_\varepsilon |u_\varepsilon|^\frac{p}{p-1}} \mathrm{d}\lambda_{\theta}.
		\end{equation*}	
		This shows that $S_{\varepsilon}(p,\nu,\theta,R)<\infty$, since $ e^{\mu_\varepsilon |u|^\frac{p}{p-1}} \in L^1_\theta $ for any  $ u \in X^{1,p}_R(\alpha,\theta)$, from \cite{JF2014}. Furthermore, $ u_\varepsilon $ is a non-negative  maximizer for $ S_{\varepsilon}(p,\nu,\theta,R) $.  We must have $ H_{\nu}(u_\varepsilon) = 1 $, otherwise, we could choose $ a > 1 $ such that $ H_{\nu}(a u_\varepsilon) = 1 $, which would lead to 
		$$
		S_{\varepsilon}(p,\nu,\theta,R) \geq \int_0^R e^{\mu_\varepsilon |a u_\varepsilon|^\frac{p}{p-1}} \mathrm{d}\lambda_{\theta} > \int_0^R e^{\mu_\varepsilon |u_\varepsilon|^\frac{p}{p-1}} \mathrm{d}\lambda_{\theta} = S_{\varepsilon}(p,\nu,\theta,R),
		$$
		which is a contradiction. By the Lagrange multipliers theorem, we get that $ u_\varepsilon $ satisfies
        \begin{equation}\label{eqlagrange}
            \int_0^R |u_\varepsilon^\prime|^{p-2} u_\varepsilon^\prime v^\prime \mathrm{d}\lambda_\alpha = \frac{1}{\lambda_\varepsilon} \int_0^R  e^{\mu_\varepsilon u_\varepsilon^\frac{p}{p-1}} u_\varepsilon^\frac{1}{p-1}v \mathrm{d}\lambda_{\theta} + \nu \int_{0}^{R}u_\varepsilon^{p-1}v \mathrm{d}\lambda_{\theta}, \;\; \mbox{for all}\;\; v \in X^{1,p}_R(\alpha,\theta),
        \end{equation}
        where  $\lambda_\varepsilon$ is given in \eqref{Lagrange}. We will proceed to show that $u_\varepsilon\in C^1[0,R] $. Following \cite{Clement-deFigueiredo-Mitidieri}, we consider the test function $ v_\rho $ given by 
		\begin{equation}\label{teste}
			v_{\rho}(r) = 
			\begin{cases}
				1, & \text{if }\; 0 \leq r \leq s \\
				1 + \frac{1}{\rho}(s-r), & \text{if } \;s \leq r \leq s + \rho \\
				0, & \text{if }\; s+\rho\le r \le R.
			\end{cases}
		\end{equation}
		By using $v_\rho$ in \eqref{eqlagrange} and letting $ \rho \to 0 $, we obtain the integral representation
		\begin{equation}
			-|u_\varepsilon^\prime|^{p-2} u_\varepsilon^\prime = \frac{1}{\omega_\alpha r^\alpha}\int_0^r \left( \frac{1}{\lambda_\varepsilon} e^{\mu_\varepsilon u_\varepsilon^\frac{p}{p-1}} u_\varepsilon^\frac{1}{p-1} + \nu u_\varepsilon^{p-1} \right) \mathrm{d} \lambda_\theta
		\end{equation}
or 
		\begin{equation}\label{defuprime}
			-u_\varepsilon^\prime(r) = \left(\frac{1}{r^\alpha \omega_\alpha} \int_0^r \Big( \frac{1}{\lambda_\varepsilon} e^{\mu_\varepsilon u_\varepsilon^\frac{p}{p-1}} u_\varepsilon^\frac{1}{p-1} + \nu u_\varepsilon^{p-1} \Big) \mathrm{d} \lambda_\theta \right)^{\frac{1}{p-1}}.
		\end{equation}
        Thus, $ u_\varepsilon \in C^1(0,R] $. In order to get the regularity at $r=0$, we first note that arguing as in \cite[Lemma~7]{CV} we can see that 
\begin{equation}\label{lemma-chave}
    \lim_{r\to 0^{+}}r^{\sigma}\left[\frac{1}{\lambda_\varepsilon} e^{\mu_\varepsilon u_\varepsilon^\frac{p}{p-1}} u_\varepsilon^\frac{1}{p-1} + \nu u_\varepsilon^{p-1} \right]=0,\;\; \mbox{for all}\;\; \sigma>0.
\end{equation}
   Recalling that we are assuming $\theta\ge \alpha$,  from L'Hospital's rule and \eqref{lemma-chave}
        \begin{align*}
          \lim_{r \to 0^+}\frac{1}{r^\alpha} \int_0^r \left( \frac{1}{\lambda_\varepsilon} e^{\mu_\varepsilon u_\varepsilon^\frac{p}{p-1}} u_\varepsilon^\frac{1}{p-1} + \nu u_\varepsilon^{p-1} \right) \mathrm{d} \lambda_\theta& =\frac{\omega_{\theta}}{\alpha} \lim_{r\to 0}r^{\theta-\alpha+1}\left[ \frac{1}{\lambda_\varepsilon} e^{\mu_\varepsilon u_\varepsilon^\frac{p}{p-1}} u_\varepsilon^\frac{1}{p-1} + \nu u_\varepsilon^{p-1} \right]=0.
        \end{align*}
	Thus, from \eqref{defuprime} we have $ u_\varepsilon^\prime(0) = 0 $, and consequently $ u_\varepsilon \in C^1[0,R] \cap X^{1,p}_R(\alpha,\theta)$.
		
     Now, we observe that 
     \begin{equation}\label{limisup<}
         \limsup_{\varepsilon\to 0} S_{\varepsilon}(p,\nu,\theta,R)  \leq S(p,\nu,\theta,R).
     \end{equation}
     On the other hand, for $ u \in X^{1,p}_R(\alpha,\theta) $ with $ H_{\nu}(u) \leq 1 $ the  Fatou's lemma ensures
		$$
		\int_0^R e^{\mu_{\alpha,\theta} |u|^\frac{p}{p-1}} \mathrm{d}\lambda_{\theta} \leq \liminf_{\varepsilon \to 0} \int_0^R e^{\mu_\varepsilon |u|^\frac{p}{p-1}} \mathrm{d}\lambda_{\theta} \le \liminf_{\varepsilon \to 0} S_{\varepsilon}(p,\nu,\theta,R)
		$$
		and thus 
        \begin{equation}\label{limiinf>}
           S(p,\nu,\theta,R)\le\liminf_{\varepsilon \to 0} S_{\varepsilon}(p,\nu,\theta,R).
        \end{equation} 
        From \eqref{limisup<} and \eqref{limiinf>}, we conclude that
		\begin{equation*}
		\lim_{\varepsilon \to 0} S_{\varepsilon}(p,\nu,\theta,R)  =S(p,\nu,\theta,R). 
		\end{equation*}	
		Finally, the elementary inequality $ e^t \le 1 + t e^t $, $ t \geq 0 $ implies
		$$
		\lambda_\varepsilon = \int_0^R e^{\mu_\varepsilon u_\varepsilon^\frac{p}{p-1}} u_\varepsilon^\frac{p}{p-1} \mathrm{d}\lambda_{\theta} \geq \int_0^R e^{\mu_\varepsilon u_\varepsilon^\frac{p}{p-1}} \mathrm{d}\lambda_{\theta} - |B_R|_{\theta} > 0.
		$$
		Hence, 
        $$ \liminf_{\varepsilon \to 0} \lambda_\varepsilon \ge S_{\varepsilon}(p,\nu,\theta,R) - |B_R|_{\theta}>0.$$
	\end{proof}
    \section{Boundedness and Extremals in the critical regime}
    The aim of this section is to prove Theorems ~\ref{teoprinc1} and  \ref{teoprinc2}. We will use a contradiction argument based on blow-up analysis and refined test-functions computations to prove both the boundedness and attainability.

    Let $u_\varepsilon \in X^{1,p}_R(\alpha,\theta) \cap C^1[0,R]$ be the sequence of subcritical maximizer constructed in Lemma~\ref{prop1}. Since $ H_{\nu}(u_\varepsilon) = 1$ and $\nu<\lambda_{\alpha,\theta}$, we have 
	\begin{equation*}
   \Big(1-\frac{\nu}{\lambda_{\alpha,\theta}}\Big)\|u^{\prime}_{\varepsilon}\|^{p}_{L^p_{\alpha}} \le \|u^{\prime}_{\varepsilon}\|^{p}_{L^p_{\alpha}}-\nu \|u_{\varepsilon}\|^p_{L^p_{\theta}}=1.
	\end{equation*}
	Hence $(u_\varepsilon)$ is bounded in $ X^{1,p}_R(\alpha,\theta)$. From \eqref{imersão}, up to a subsequence, we can write 
    \begin{equation}\label{convergencias}
		u_\varepsilon \rightharpoonup u_0 \mbox{ in } X^{1,p}_R(\alpha,\theta), \quad u_\varepsilon \to u_0 \mbox{ in } L^{q}_\theta\; (q>1) \quad \mbox{ and } u_\varepsilon(r) \to u_0(r) \mbox{ a.e. in } (0,R).
	\end{equation}
   Recall that $u_{\varepsilon}$ is a decreasing function. Then we set
	\begin{equation*}
	a_\varepsilon = \max_{[0,R]} u_\varepsilon(r) = u_\varepsilon(0).
	\end{equation*} 
    Our analysis is divided into two cases:

\noindent \hypertarget{Case 1}{(C1)} $u_0 \not\equiv 0$ or $(a_\varepsilon)$ is bounded.

\noindent \hypertarget{Case 2}{(C2)} (blow-up) $u_0 \equiv 0$ and $a_\varepsilon \to +\infty$ as $\varepsilon \to 0$.

We will show that  \hyperlink{Case 1}{(C1)}  yields the desired result, whereas \hyperlink{Case 2}{(C2)}  cannot occur (it is impossible). Firstly, we prove the following:
	\begin{lemma}\label{Theo3}
		Suppose \hyperlink{Case 1}{(C1)} holds. Then $u_0 \in C^1[0,R]$, $H_{\nu}(u_0)\le 1$ and
		\begin{equation}
			S(p,\nu,\theta,R)=\int_0^R e^{\mu_{\alpha,\theta} u_0^\frac{p}{p-1}} \mathrm{d}\lambda_{\theta}.
		\end{equation}
	\end{lemma}
	\begin{proof}
		From \eqref{convergencias}, we have $H_{\nu}(u_0)\le 1$. If $u_0 \not \equiv 0 $, then by Lemma~\ref{improvement lions} there exists $1<q<P_{H}(u_0)$ such that	
		\begin{equation}\label{U-integrably}
		    \limsup_{\epsilon\to0}\int_0^R e^{q\mu_\varepsilon u_\varepsilon^\frac{p}{p-1}} \mathrm{d}\lambda_{\theta}<\infty.
		\end{equation}
		By \eqref{convergencias}, we have  $\exp({\mu_\varepsilon u_\varepsilon^\frac{p}{p-1}}) \to \exp({\mu_{\alpha,\theta} u_0^\frac{p}{p-1}})$ a.e. in $(0,R)$. It follows from Vitali's convergence theorem 
		\begin{equation}\label{L-resultado}
		   S(p,\nu,\theta,R)= \lim_{\varepsilon \to 0}S_{\varepsilon}(p,\nu,\theta,R)=\lim_{\varepsilon \to 0}\int_0^R e^{\mu_\varepsilon u_\varepsilon^\frac{p}{p-1}} \mathrm{d}\lambda_{\theta}=\int_0^R e^{\mu_{\alpha,\theta} u_0^\frac{p}{p-1}} \mathrm{d}\lambda_{\theta},
		\end{equation}
        where we also have used Lemma~\ref{prop1}.        
	On the other hand, if $(a_\varepsilon)$ is bounded, then we have $|u_{\varepsilon}|\le a_\varepsilon \le c$ for any $\varepsilon>0$. Hence, \eqref{U-integrably} holds for any $q>1$. So, we can use the Vitali's convergence theorem  to get \eqref{L-resultado}. Finally, 
    by using Lagrange multipliers, we have 
		\begin{equation}\label{Lagrangeint}
			\int_0^R|u_0^\prime|^{p-2}u_0^\prime v^\prime \mathrm{d}\lambda_{\alpha}=\int_0^R\frac{1}{\lambda_0}e^{\mu_{\alpha\theta}u_0^\frac{p}{p-1}}u_0^\frac{1}{p-1}\mathrm{d}\lambda_{\theta}+\nu \int_0^Ru_0^{p-1}v\mathrm{d}\lambda_{\theta},\;\;\mbox{for all}\;\; v\in X^{1,p}_R(\alpha,\theta).
		\end{equation}
	Hence, by using the same argument as in Lemma~\ref{prop1}, it follows that $u_0 \in C^1[0,R]$. 
\end{proof}

By Lemma~\ref{Theo3}, we only need to analyze   \hyperlink{Case 2}{(C2)}. Hereafter we shall assume the following:
\begin{equation}\label{Ablow-up}
u_0\equiv0\;\;\mbox{and}\;\; a_{\varepsilon}\to +\infty, \;\; \mbox{as}\;\; \varepsilon\to 0.
\end{equation}
\begin{lemma}\label{concentracao} If  $u_0\equiv 0$, then $(u_{\varepsilon})$ is a
concentrating sequence at  the origin, that is, 
$$
H_\nu(u_{\varepsilon})= 1,\quad
u_{\varepsilon}\rightharpoonup 0\quad\mbox{in}\quad X^{1,p}_R(\alpha,\theta),\quad\mbox{and}\quad
\lim_{\varepsilon\to 0}\int_{r_0}^{R}|u^{\prime}_{\varepsilon}|^{p}\,\mathrm{d}\lambda_{\alpha}=0,\;\forall
r_0>0.
$$
\end{lemma}
\begin{proof}
Recall that $\int_0^R|u_\varepsilon^\prime|^{p}\mathrm{d}\lambda_\alpha=1+\nu\int_0^R|u_\varepsilon|^{p}\mathrm{d}\lambda_{\theta} \to 1$, as $\varepsilon\to 0$. We argue by contradiction. Assume that there exist constants $a<1$ and $r_0>0$ such that
$$
\lim_{\varepsilon\to 0}\int_{r_0}^{R} |u'_\varepsilon|^{p}\,\mathrm{d}\lambda_{\alpha}>a.
$$
Since $\int_0^R|u_\varepsilon^\prime|^{p}\mathrm{d}\lambda_\alpha\to 1$, we have
\begin{equation}\label{1-a}
\lim_{\varepsilon\to 0}\int_{0}^{r_0} |u'_\varepsilon|^{p}\,\mathrm{d}\lambda_{\alpha}< 1-a.
\end{equation}
Let us take an auxiliary function $\varphi \in C^1[0,R]$, $0 \leq \varphi \leq 1$ such that $\varphi\equiv 1$ in $[0,\frac{r_0}{2}]$ and $\varphi\equiv 0$ in $[r_0,R]$. We have
\begin{equation}{\label{limitation}}
    \begin{aligned}
        \|(\varphi u_\varepsilon)^\prime\|_{L^{p}_{\alpha}}\leq \|\varphi u^\prime_\varepsilon\|_{L^{{p}}_\alpha}+\|\varphi^\prime u_\varepsilon\|_{L^{p}_\alpha}
        \leq \Big(\int_0^{r_0}|u^\prime_\varepsilon|^{p}\mathrm{d}\lambda_\alpha\Big)^{\frac{1}{p}}+\|\varphi^\prime u_\varepsilon\|_{L^{p}_\alpha}.
    \end{aligned}
\end{equation}
By the compact embedding \eqref{imersão}, and using \eqref{convergencias}, \eqref{1-a}, and \eqref{limitation}, we obtain
$$
\limsup_{\varepsilon\to 0}\|(\varphi u_\varepsilon)^\prime\|^{p}_{L^{p}_\alpha}<1-a.
$$
We can choose $q>1$ such that
$$
q\|(\varphi u_\varepsilon)^\prime\|_{L^{p}_\alpha}^{\frac{p}{p-1}}<(1-a)^{\frac{1}{p-1}}<1,
$$
for $\varepsilon>0$ small enough.
Hence, by \eqref{TM1}, for $\varepsilon$ small enough.
\begin{equation}\label{U-int-parte1}
    \begin{aligned}
        \int_0^{\frac{r_0}{2}}e^{q\mu_{\varepsilon}|u_\varepsilon|^\frac{p}{p-1}}\mathrm{d}\lambda_{\theta}&=\int_0^{\frac{r_0}{2}}e^{q\mu_{\varepsilon} \|(\varphi u_\varepsilon)^\prime\|_{L^{p}_\alpha}^{\frac{p}{p-1}}\Big|\frac{|\varphi u_\varepsilon|}{\|(\varphi u_\varepsilon)^\prime\|_{L_\alpha^{p}}}\Big|^\frac{p}{p-1}}\mathrm{d}\lambda_{\theta}\\
        &\leq \int_0^R e^{\mu_{\alpha,\theta}\Big|\frac{|\varphi u_\varepsilon|}{\|(\varphi u_\varepsilon)^\prime\|_{L_\alpha^{p}}}\Big|^\frac{p}{p-1}} \mathrm{d}\lambda_{\theta} \leq c_1
    \end{aligned}
\end{equation}
where  $c_1>0$ is independent of $\varepsilon.$

Note that
\begin{eqnarray*}
|u_\varepsilon(r)|\leq \int^R_r|u^\prime_\varepsilon(s)|\mathrm{d} s &=&\int_r^R(\omega_\alpha^{\frac{1}{p}}s^\frac{\alpha}{p}|u_\varepsilon^\prime(s)|)(\omega_\alpha^{-\frac{1}{p}}s^{-\frac{\alpha}{p}}) \mathrm{d} s\\
&\leq& \|u^\prime_\varepsilon\|_{L^{p}_\alpha}\Big(\omega_\alpha^{-\frac{1}{\alpha}}\ln\frac{R}{r}\Big)^\frac{p-1}{p}\\
&=& \|u^\prime_\varepsilon\|_{L^{p}_\alpha}\Big(\frac{\theta+1}{\mu_{\alpha,\theta}}\ln\frac{R}{r}\Big)^\frac{p-1}{p},
\end{eqnarray*}
for all $0<r\le R$.
Using that $\|u_\varepsilon^\prime\|_{L^{p}_\alpha} \to 1$, we obtain
\begin{equation}\label{limita}
\mu_{\alpha,\theta}|u_\varepsilon(r)|^\frac{p}{p-1}\leq 2(\theta+1) \ln \frac{2R}{r_0},\;\;\mbox{for all} \;\; \frac{r_0}{2}\le r\le R
\end{equation}
if $\varepsilon>0$ is sufficiently small. It follows from \eqref{limita} that
\begin{equation}\label{U-int-parte2}
    \int_{\frac{r_0}{2}}^Re^{q \mu_{\varepsilon}|u_\varepsilon|^\frac{p}{p-1}}\mathrm{d}\lambda_{\theta} \leq \Big(\frac{2R}{r_0}\Big)^{2q(\theta+1)}\int_{\frac{r_0}{2}}^R\mathrm{d}\lambda_{\theta}.
\end{equation}
By combining \eqref{U-int-parte1} and \eqref{U-int-parte2}, we are in a position to apply the Vitali convergence theorem to get
$$
S(p,\nu,\theta,R)=\lim_{\epsilon\to 0}S_{\varepsilon}(p,\nu,\theta,R)=\lim_{\epsilon\to 0}\int_0^Re^{\mu_\varepsilon|u_{\varepsilon}|^\frac{p}{p-1}}\mathrm{d}\lambda_{\theta}= \int_0^R \mathrm{d}\lambda_{\theta}=|B_R|_{\theta}
$$
which leads to a contradiction.

\end{proof}
To exclude case \hyperlink{Case 2}{(C2)}, we proceed in two steps. First, we apply a blow-up analysis to show that, under condition \eqref{Ablow-up}, one must have (cf.\ Lemma~\ref{lemacontrad} below)
\begin{equation}\nonumber
	S(p,\nu,\theta,R)\leq |B_R|_\theta + e^{\mu_{\alpha,\theta}A_{0}+ \Psi(p)+\gamma}|B_1|_\theta.
\end{equation}
On the other hand, by testing with a suitable function (cf.\ Lemma~\ref{Lemma-TF} below), we show that
\begin{equation}\nonumber
	S(p,\nu,\theta,R)> |B_R|_\theta + e^{\mu_{\alpha,\theta}A_{0}+ \Psi(p)+\gamma}|B_1|_\theta.
\end{equation}
This contradiction implies that \hyperlink{Case 2}{(C2)} cannot occur.
\subsection{Blow-up analysis}
We shall analyze the behavior of the sequence $(u_{\varepsilon})$ given by Lemma~\ref{prop1} around the blow-up point $ r=0$. To this end, let us define the auxiliary functions
	\begin{equation}\label{phipsi}
		\left\{\begin{aligned}
	&\varphi_\varepsilon(r)=\frac{u_\varepsilon(r_\varepsilon r)}{a_\varepsilon}\\
		&	\psi_\varepsilon(r)=a_\varepsilon^{\frac{1}{p-1}}(u_\varepsilon(r_\varepsilon r)-a_\varepsilon)
		\end{aligned},\right. \;\;\mbox{for}\;\;\ 0<r\le \frac{R}{r_{\varepsilon}}
	\end{equation}
	where 
	\begin{equation}\label{rE}
		r_\varepsilon^{\theta+1}= \frac{\lambda_\varepsilon}{a_\varepsilon^{\frac{p}{p-1}}e^{\mu_\varepsilon a_\varepsilon^{\frac{p}{p-1}}}}.
	\end{equation}
	\begin{lemma}\label{limr} Let $\eta<\mu_{\alpha,\theta}$. Then
		the sequence $(r_\varepsilon)$ satisfies 
         \begin{equation*}
     r_\varepsilon^{\theta+1}a_\varepsilon^\frac{p}{p-1} e^{\eta a_\varepsilon^\frac{p}{p-1}} \to 0, \;\; \mbox{as}\;\;\varepsilon\to 0.
 \end{equation*}
In particular,  $r_\varepsilon^{\theta+1} \to 0$ when $\varepsilon \to 0$.
	\end{lemma}
	
	\begin{proof}
		For $\varepsilon>0$ sufficiently small, it follows that $\eta<\mu_\varepsilon$ which implies  
$$
(\mu_\varepsilon-\eta)\Big(u_\varepsilon^{\frac{p}{p-1}} - a_\varepsilon^{\frac{p}{p-1}}\Big) \leq 0.
$$  
So, from the definition of $r_\varepsilon$, we can write

\begin{equation*}
\begin{aligned}
r_\varepsilon^{\theta+1}a_\varepsilon^\frac{p}{p-1} e^{\eta a_\varepsilon^\frac{p}{p-1}}&=\int_0^Re^{\mu_\varepsilon u_\varepsilon^\frac{p}{p-1}}u_\varepsilon^\frac{p}{p-1}e^{{\eta a_\varepsilon^\frac{p}{p-1}}-\mu_\varepsilon a_\varepsilon^\frac{p}{p-1}}\mathrm{d}\lambda_{\theta} \\
&=\int_0^R u_\varepsilon^\frac{p}{p-1}e^{(\mu_\varepsilon-\eta)(u_\varepsilon^\frac{p}{p-1}-a_\varepsilon^\frac{p}{p-1})+\eta u_\varepsilon^{\frac{p}{p-1}}}\mathrm{d}\lambda_{\theta}\\
&\leq  \int_0^R u_\varepsilon^\frac{p}{p-1}e^{\eta u_\varepsilon^{\frac{p}{p-1}}} \mathrm{d}\lambda_{\theta}.
\end{aligned}
\end{equation*}
		Let $\bar{q}>1$ be chosen such that $\bar{q} \eta \leq \mu_{\alpha,\theta}$. By H\"{o}lder inequality and the Trudinger-Moser inequality in \cite{JF2014}, we have 
		\begin{align*}
		    \int_0^R u_\varepsilon^\frac{p}{p-1}e^{\eta u_\varepsilon^{\frac{p}{p-1}}} \mathrm{d}\lambda_{\theta} & \leq \Big(\int_0^R e^{\bar{q}\eta u_\varepsilon^{\frac{p}{p-1}}} \mathrm{d}\lambda_{\theta} \Big)^\frac{1}{\bar{q}} \Big(\int_0^R u_\varepsilon^{\frac{\bar{q}p}{(p-1)(\bar{q}-1)}} \mathrm{d}\lambda_{\theta} \Big)^{\frac{\bar{q}-1}{\bar{q}}}\\
            &\le c\Big(\int_0^R u_\varepsilon^{\frac{\bar{q}p}{(p-1)(\bar{q}-1)}} \mathrm{d}\lambda_{\theta} \Big)^{\frac{\bar{q}-1}{\bar{q}}}\to 0,\;\; \mbox{as}\;\; \varepsilon\to 0
		\end{align*}
where we use that $u_\varepsilon \to u_0 \equiv 0$ in $L^q_{\theta}$ for every $q>1$ (cf. \eqref{convergencias}). Thus,
 \begin{equation*}
     r_\varepsilon^{\theta+1}a_\varepsilon^\frac{p}{p-1} e^{\eta a_\varepsilon^\frac{p}{p-1}} \to 0,\;\; \mbox{as}\;\; \varepsilon\to 0
 \end{equation*}
 as desired.
	\end{proof}
    \begin{remark}
        Since $r_\varepsilon \to 0$ as $\varepsilon\to 0$, for any $s>0$, there exists $\varepsilon > 0$ such that $s < \frac{R}{r_\varepsilon}$ for sufficiently small $\varepsilon$. 
    \end{remark}
 We now investigate the limiting behavior of $\psi_\varepsilon$ and $\varphi_\varepsilon$ given in \eqref{phipsi}, as $\varepsilon \to 0$.

	\begin{lemma}\label{to1}
		We have $\varphi_\varepsilon \to 1$ in $C^1_{loc}[0,\infty)$.
	\end{lemma}
	\begin{proof}
		By \eqref{Lagrange} we get
		\begin{equation*}
			\int_0^{\frac{R}{r_\varepsilon}}|\varphi_\varepsilon^\prime|^{p-2}\varphi_\varepsilon^\prime z^\prime\mathrm{d}\lambda_\alpha=\frac{1}{a_\varepsilon^{p}}\int_0^{\frac{R}{r_\varepsilon}}e^{\mu_\varepsilon \big(u_\varepsilon^\frac{p}{p-1}(r_{\varepsilon }s)-a_\varepsilon^{\frac{p}{p-1}}\big)}\varphi_\varepsilon^\frac{1}{p-1}z\mathrm{d}\lambda_{\theta}+r_\varepsilon^{\theta+1}\nu  \int_0^{\frac{R}{r_\varepsilon}}\varphi_\varepsilon^{p-1}z\mathrm{d}\lambda_{\theta}
		\end{equation*}
		for all $z(s)=v(sr_\varepsilon) $, with $v \in X^{1,p}_R(\alpha,\theta)$. By choosing the test function $v_\rho$ as in \eqref{teste}, letting $\rho \to 0$, we have

		\begin{equation}\label{vline}
			\omega_\alpha |\varphi^\prime_\varepsilon(r)|^{p-1}
            =\frac{1}{r^\alpha}\frac{1}{a_\varepsilon^{p}}\int_0^{r}e^{\mu_\varepsilon \big(u_\varepsilon^\frac{p}{p-1}(r_{\varepsilon}s)-a_\varepsilon^{\frac{p}{p-1}}\big)}\varphi_\varepsilon^\frac{1}{p-1}\mathrm{d}\lambda_{\theta}
            +r_\varepsilon^{\theta+1}\nu \frac{1}{r^\alpha} \int_0^{r}\varphi_\varepsilon^{p-1}\mathrm{d}\lambda_{\theta}.
		\end{equation}
		Recall that $\theta \geq \alpha$,  $u_\varepsilon^{\frac{p}{p-1}} - a_\varepsilon^{\frac{p}{p-1}} \leq 0$ and $\varphi_\varepsilon \leq 1$. Since $a_\varepsilon \to \infty$ and $r_\varepsilon \to 0$, the equation \eqref{vline} yields $\varphi_\varepsilon^{\prime} \to 0$ uniformly on $[0, r_0]$. Given that $\varphi_\varepsilon(0) = 1$, we deduce that $\varphi_\varepsilon \to 1$ in $C^1_{\text{loc}}[0, \infty)$.

	\end{proof}
	
	\begin{lemma}\label{limpsi}
		We have  $\psi_\varepsilon \to \psi$ in $C^1_{loc}[0,\infty)$, where 
		$$
		\psi(r)=-\frac{p-1}{\mu_{\alpha,\theta}} \ln \Big(1+c_0r^{\frac{\theta+1}{p-1}}\Big),\;\;\mbox{with}\;\; c_0=\Big(\frac{\omega_{\theta}}{\theta+1}\Big)^{\frac{1}{\alpha}}.
		$$
	\end{lemma}
	\begin{proof}
		Arguing as in \eqref{vline}, we can write
		\begin{equation}\label{wline}
			\omega_\alpha |\psi^\prime_\varepsilon(r)|^{p-1}=\frac{1}{r^\alpha}\int_0^re^{\mu_\varepsilon \big(u_\varepsilon^\frac{p}{p-1}(r_{\varepsilon}s)-a_\varepsilon^{\frac{p}{p-1}}\big)}\varphi_\varepsilon^\frac{1}{p-1}\mathrm{d}\lambda_{\theta}
            +a_\varepsilon r_\varepsilon^{\theta+1}\nu\frac{1}{r^\alpha}\int_{0}^r\varphi_\varepsilon^{p-1}\mathrm{d}\lambda_{\theta}
		\end{equation}
        and thus 
        \begin{equation}\label{wprime-line}
			-\psi^\prime_\varepsilon(r)=\Bigg(\frac{1}{\omega_{\alpha}r^\alpha}\int_0^re^{\mu_\varepsilon \big(u_\varepsilon^\frac{p}{p-1}(r_{\varepsilon}s)-a_\varepsilon^{\frac{p}{p-1}}\big)}\varphi_\varepsilon^\frac{1}{p-1}\mathrm{d}\lambda_{\theta}
			+\frac{a_\varepsilon r_\varepsilon^{\theta+1}\nu}{\omega_{\alpha}}\frac{1}{r^\alpha}\int_{0}^r\varphi_\varepsilon^{p-1}\mathrm{d}\lambda_{\theta}\Bigg)^{\frac{1}{p-1}}.
		\end{equation}
		By Lemma~\ref{limr} and Lemma~\ref{to1}, we have   $a_\varepsilon r_\varepsilon^{\theta+1}\to 0$ and $\varphi_{\varepsilon}\to 1$ in $C^1_{loc}[0,\infty)$. Thus, given $r_0>0$ in \eqref{wline} we get $\psi_\varepsilon^\prime$ is bounded in $C[0,r_0]$. Since $\psi_\varepsilon(0)=0$  for all $\varepsilon>0$, we get $\psi_\varepsilon$ is a uniformly equicontinuos family in $C[0,r_0]$. By Ascoli-Arzelà theorem we obtain $\psi_\varepsilon \to w$ in $C[0,r_0]$. Since $r_0$ is arbitrary, we have $\psi_\varepsilon \to \psi$ in $C^0_{\text{loc}}[0, \infty)$. In addition, from \eqref{phipsi}
		\begin{equation}\label{taylor}
			u_\varepsilon(r_\varepsilon s)^{\frac{p}{p-1}}-a_\varepsilon^{\frac{p}{p-1}}=\frac{p}{p-1} \psi_\varepsilon(s)(1+O(\varphi_{\varepsilon}-1)).
		\end{equation}
	By integrating in \eqref{wprime-line} on the interval $(0, r)$, we conclude that
    \begin{equation}\nonumber
			\psi_\varepsilon(r)=-\int_{0}^{r}\Big(\frac{1}{\omega_{\alpha}t^{\alpha}}\int_0^t\Big(e^{\mu_\varepsilon (u_\varepsilon^\frac{p}{p-1}-a_\varepsilon^{\frac{p}{p-1}})}\varphi_\varepsilon^\frac{1}{p-1}
            +a_\varepsilon r_\varepsilon^{\theta+1}\nu\varphi_\varepsilon^{p-1}\Big)\mathrm{d}\lambda_{\theta}\Big)^{\frac{1}{p-1}}\mathrm{d} t.
		\end{equation}
		Letting $\varepsilon \to 0$  the dominated convergence theorem implies
		$$
		\psi(r)=- \int _0^r\left( \frac{1}{\omega_\alpha t^\alpha} \int_0^t e^{\frac{p}{p-1}\mu_{\alpha,\theta} \psi}\mathrm{d}\lambda_{\theta}\right)^\frac{1}{p-1} \mathrm{d} t.
		$$
        In particular, 
        \begin{equation}\label{Wprime-line}
            \psi^{\prime}(r)=-\left( \frac{1}{\omega_\alpha r^\alpha} \int_0^r e^{\frac{p}{p-1}\mu_{\alpha,\theta} \psi}\mathrm{d}\lambda_{\theta}\right)^\frac{1}{p-1} \mathrm{d} t.
        \end{equation}
	Further, $\psi$ satisfies the  the differential equation 
		\begin{equation*}
			\begin{cases}
				-\omega_\alpha(r^\alpha|\psi^\prime|^{p-2} \psi^\prime)^\prime=\omega_\theta r^\theta e^{\frac{p}{p-1}\mu_{\alpha,\theta}\psi} 	\mbox{ on } [0,\infty)\\
				\psi(0)=0 \mbox{ and } 
				 \psi^\prime(0)=0.
			\end{cases}
		\end{equation*}
	By uniqueness of solutions, we obtain the desired expression for $\psi$. Finally, by comparing \eqref{wprime-line} and \eqref{Wprime-line}, and using Lemma~\ref{limr} and Lemma~\ref{to1} we can see that $\psi^{\prime}_{\epsilon}\to \psi^{\prime}$ in $C^0_{loc}[0,\infty)$.
	\end{proof}

	\begin{lemma}\label{lemmaequal1}
		The function $\psi$ in Lemma~\ref{limpsi} satisfies  
		\begin{equation*}
			\int_0^\infty e^{\frac{p}{p-1}\mu_{\alpha,\theta}\psi} \mathrm{d}\lambda_{\theta}=1.
		\end{equation*}
    
	\end{lemma}
	\begin{proof}
    Let $\Gamma$ be the Gamma Euler function given in \eqref{fractional volume}. We recall the following  properties:
    \begin{equation*}
      \Gamma(1)=1,\;\; \Gamma(x+1)=x\Gamma(x) \;\;\mbox{and}\;\;  \int_0^\infty \frac{s^{x-1}}{(1+s)^{x+y}}=\frac{\Gamma(x)\Gamma(y)}{\Gamma(x+y)}, \ x,y>0,
    \end{equation*}
    see for instance \cite{S-functions} for details. Thus, by
using the change variables
$s=c_{0}r^{\frac{\theta+1}{p-1}}$    we obtain 
\begin{equation}\label{integral-cabulosa}
\int_{0}^{\infty}\mathrm{e}^{\frac{p}{p-1}\mu_{\alpha,\theta}\psi(r)}\,\mathrm{d}\lambda_{\theta}
=\int_{0}^{\infty}\frac{\mathrm{d}\lambda_{\theta}}{\left(1+c_0r^{\frac{\theta+1}{p-1}}\right)^{p}}
=(p-1)\frac{\Gamma(p-1)\Gamma(1)}{\Gamma(p)}=1.
\end{equation}
	\end{proof}
	
\begin{lemma}\label{uepsilon} For $c>1$, let $u_{\varepsilon,c}=\min\{u_\varepsilon, \frac{a_\varepsilon}{c}\}$.  Then,
		\begin{enumerate}
			\item[$i)$] $\displaystyle\lim_{\varepsilon \to 0}\|u_{\varepsilon,c}^\prime\|^{p}_{L^{p}_\alpha} = \frac{1}{c}$,\\
            
			\item[$ii)$] $\displaystyle\lim_{\varepsilon \to 0}\|(u_\varepsilon -u_{\varepsilon,c})^\prime\|^{p}_{L^{p}_\alpha} = \frac{c-1}{c}$.
		\end{enumerate}
	\end{lemma}
	\begin{proof}	
	 Taking $v = u_{\varepsilon,c}$ in \eqref{Lagrange}, we can write 
		\begin{equation*}
        \begin{aligned}
			\int_0^R|(u_{\varepsilon,c})^\prime|^{p} \mathrm{d}\lambda_\alpha
            =\frac{1}{\lambda_\varepsilon}\int_0^Re^{\mu_\varepsilon u_\varepsilon^\frac{p}{p-1}}u_\varepsilon^\frac{1}{p-1}u_{\varepsilon,c}\mathrm{d}\lambda_{\theta}+\nu\int_0^R u_\varepsilon^{p-1}u_{\varepsilon,c}\mathrm{d}\lambda_{\theta}.
        \end{aligned}
	\end{equation*}
    Since $u_\varepsilon \to 0$ in $L^p_\theta$ for any $p>1$, we obtain
\begin{equation}\label{grad-step1}
        \begin{aligned}
			\int_0^R|(u_{\varepsilon,c})^\prime|^{p} \mathrm{d}\lambda_\alpha
			&=\frac{1}{\lambda_\varepsilon}\int_0^Re^{\mu_\varepsilon u_\varepsilon^\frac{p}{p-1}}u_\varepsilon^\frac{1}{p-1}u_{\varepsilon,c} \mathrm{d}\lambda_{\theta}+o_\varepsilon(1).
        \end{aligned}
	\end{equation}
Now, for $r_0>0$,  Lemma~\ref{to1} ensures that $[0, r_0 r_\varepsilon) \subset \left\{ u_\varepsilon > a_\varepsilon/c \right\}$,  for $\varepsilon>0$ small enough.  Hence
\begin{equation}\label{grad-step2}
        \begin{aligned}
			\frac{1}{\lambda_\varepsilon}\int_0^Re^{\mu_\varepsilon u_\varepsilon^\frac{p}{p-1}}u_\varepsilon^\frac{1}{p-1}u_{\varepsilon,c} \mathrm{d}\lambda_{\theta}
			&=\frac{1}{\lambda_\varepsilon}\int_{\{u_\varepsilon\leq\frac{a_\varepsilon}{c}\}}e^{\mu_\varepsilon u_\varepsilon^\frac{p}{p-1}}u_\varepsilon^\frac{1}{p-1}u_{\varepsilon,c} \mathrm{d}\lambda_{\theta}+\frac{1}{\lambda_\varepsilon}\int_{\{u_\varepsilon>\frac{a_\varepsilon}{c}\}}e^{\mu_\varepsilon u_\varepsilon^\frac{p}{p-1}}u_\varepsilon^\frac{1}{p-1}u_{\varepsilon,c} \mathrm{d}\lambda_{\theta}\\
			&\ge \frac{a_\varepsilon}{c}\frac{1}{\lambda_\varepsilon}\int_0^{r_\varepsilon r_0}e^{\mu_\varepsilon u_\varepsilon^\frac{p}{p-1}}u_\varepsilon^\frac{1}{p-1} \mathrm{d}\lambda_{\theta}.
        \end{aligned}
	\end{equation}
Setting $s = r_\varepsilon r$ and using \eqref{taylor} as $\varepsilon \to 0$, we obtain

\begin{equation}\label{grad-step3}
        \begin{aligned}
			\frac{a_\varepsilon}{c}\int_0^{r_\varepsilon r_0}\frac{1}{\lambda_\varepsilon}e^{\mu_\varepsilon u_\varepsilon^\frac{p}{p-1}}u_\varepsilon^\frac{1}{p-1} \mathrm{d}\lambda_{\theta}&=\frac{1}{c}\int_0^{r_0}e^{\mu_\varepsilon a_\varepsilon^\frac{p}{p-1}(\varphi^\frac{p}{p-1}_\varepsilon-1)}\varphi_\varepsilon^\frac{1}{p-1} \mathrm{d}\lambda_{\theta}+o_\varepsilon(1)\\
			&\to\frac{1}{c}\int_0^{r_0} e^{\frac{p}{p-1}\mu_{\alpha,\theta}\psi} \mathrm{d}\lambda_{\theta},\;\;\mbox{as}\;\;\varepsilon\to 0.
            \end{aligned}
		\end{equation}
By letting $r_0 \to \infty$ and using Lemma~\ref{lemmaequal1}, the estimates \eqref{grad-step1},\eqref{grad-step2} and \eqref{grad-step3} yield
 $$
 \liminf_{\varepsilon \to 0}\|u_{\varepsilon,c}^\prime\|^{p}_{L^{p}_\alpha}\geq\frac{1}{c}.
 $$
It is easy to check that $u_{\varepsilon}-u_{\varepsilon,c}=(u_\varepsilon-\frac{a_\varepsilon}{c})^{+}$. By taking the test function $v=(u_\varepsilon-\frac{a_\varepsilon}{c})^{+}$ in \eqref{Lagrange} we have 		
		\begin{equation}\label{u-c=step1}
        \begin{aligned}
		\int_0^R\Big|(u_\varepsilon-u_{\varepsilon,c})^\prime\Big|^{p} \mathrm{d}\lambda_\alpha
            =&\frac{1}{\lambda_\varepsilon}\int_0^R e^{\mu_\varepsilon u_\varepsilon^\frac{p}{p-1}}u_\varepsilon^\frac{1}{p-1}\Big(u_\varepsilon-\frac{a_\varepsilon}{c}\Big)^{+}\mathrm{d}\lambda_{\theta}+\nu\int_{0}^{R} u_\varepsilon^{p-1}\Big(u_\varepsilon-\frac{a_\varepsilon}{c}\Big)^{+}\mathrm{d}\lambda_{\theta}\\
			 \geq & \frac{1}{\lambda_\varepsilon}\int_0^Re^{\mu_\varepsilon u_\varepsilon^\frac{p}{p-1}}u_\varepsilon^\frac{1}{p-1}\Big(u_\varepsilon-\frac{a_\varepsilon}{c}\Big)^{+} \mathrm{d}\lambda_{\theta}.
            \end{aligned}
		\end{equation}
        Now, by setting $s = r_\varepsilon r$ we have
        \begin{equation}\label{u-c=step2}
        \begin{aligned}
		\frac{1}{\lambda_\varepsilon}\int_0^Re^{\mu_\varepsilon u_\varepsilon^\frac{p}{p-1}}u_\varepsilon^\frac{1}{p-1}\Big(u_\varepsilon-\frac{a_\varepsilon}{c}\Big)^{+} \mathrm{d}\lambda_{\theta}&=\frac{1}{\lambda_\varepsilon}\int_{\{u_\varepsilon>\frac{a_\varepsilon}{c}\}}e^{\mu_\varepsilon u_\varepsilon^\frac{p}{p-1}}u_\varepsilon^\frac{1}{p-1}\Big(u_\varepsilon-\frac{a_\varepsilon}{c}\Big)^{+} \mathrm{d}\lambda_{\theta}\\
			&\ge \frac{1}{\lambda_\varepsilon}\int_0^{r_\varepsilon r_0}e^{\mu_\varepsilon u_\varepsilon^\frac{p}{p-1}}u_\varepsilon^\frac{1}{p-1}\Big(u_\varepsilon-\frac{a_{\varepsilon}}{c}\Big)^{+} \mathrm{d}\lambda_{\theta}\\
			&=\int_0^{r_0}e^{\mu_\varepsilon a_\varepsilon^\frac{p}{p-1}(\varphi^\frac{p}{p-1}_\varepsilon-1)}\varphi_\varepsilon^\frac{1}{p-1}\Big(\varphi_\varepsilon-\frac{1}{c}\Big)^{+} \mathrm{d}\lambda_{\theta}.
            \end{aligned}
		\end{equation}
		By combining Lemma~\ref{to1} and Lemma~\ref{limpsi}  with \eqref{taylor}, \eqref{u-c=step1} and \eqref{u-c=step2}, we have
		$$
		\liminf_{\varepsilon\to 0}\int_0^R|(u_\varepsilon-u_{\varepsilon,c})^\prime|^{p} \mathrm{d}\lambda_\alpha\geq \frac{c-1}{c}\int_{0}^{r_0} e^{\frac{p}{p-1}\mu_{\alpha,\theta}\psi} \mathrm{d}\lambda_{\theta}.
		$$
		Letting $r_0 \to \infty $ and using Lemma~\ref{lemmaequal1}, we get
		$$
		\liminf_{\varepsilon\to 0}\int_0^R|(u_\varepsilon-u_{\varepsilon,c})^\prime|^{p} \mathrm{d}\lambda_\alpha\geq \frac{c-1}{c}.
		$$
Now, observe that
		\begin{equation}\label{equality}
			\|u_{\varepsilon,c}^\prime\|^{p}_{L^{p}_\alpha}=\|u^\prime_\varepsilon\|^{p}_{L^{p}_\alpha}-\|(u_\varepsilon-u_{\varepsilon,c})^\prime\|^{p}_{L^{p}_\alpha}=1+\nu\|u_\varepsilon\|^{p}_{L^{p}_\theta}-\|(u_\varepsilon-u_{\varepsilon,c})^\prime\|^{p}_{L^{p}_\alpha}.
		\end{equation}
		Letting $\varepsilon \to 0$, we obtain $\limsup_{\varepsilon\to0}\|u_{\varepsilon,c}^\prime\|^{p}_{L^{p}_\alpha} \leq \frac{1}{c}$. Therefore, 

		$$\
		\lim_{\varepsilon \to 0}\|(u_{\varepsilon,c})^\prime\|^{p}_{L^{p}_\alpha} = \frac{1}{c}.
		$$ 
		By \eqref{equality} and using $(i)$, we have 
		$$
		\lim_{\varepsilon \to 0}\|(u_\varepsilon-u_{\varepsilon,c})^\prime\|^{p}_{L^{p}_\alpha}=\frac{c-1}{c}.
		$$
	\end{proof}
	
	\begin{lemma}\label{desig}
	It holds 
		$$
		S(p,\nu,\theta,R) =\lim_{\varepsilon \to 0} S_{\varepsilon}(p,\nu,\theta,R)  \leq |B_R|_{\theta}+\limsup_{\varepsilon \to 0}\frac{\lambda_\varepsilon}{a_\varepsilon^\frac{p}{p-1}}.
		$$
	\end{lemma}
	\begin{proof}
		The first identity was established in Lemma~\ref{prop1}. Furthermore, we may assume that 
        $$\displaystyle\limsup_{\varepsilon \to 0} \frac{\lambda_\varepsilon}{a_\varepsilon^{\frac{p}{p-1}}} < \infty,$$
       otherwise there is nothing to prove. Since $u_\varepsilon \geq a_\varepsilon/c$ in $\{u_\varepsilon> a_\varepsilon\slash c\}$, we can write
		\begin{eqnarray*}
			 S_{\varepsilon}(p,\nu,\theta,R)&=&\int_{\{u_\varepsilon \leq a_\varepsilon/ c\}}e^{\mu_\varepsilon u_\varepsilon^\frac{p}{p-1}} \mathrm{d}\lambda_{\theta}+\int_{\{u_\varepsilon > a_\varepsilon / c\}}e^{\mu_\varepsilon u_\varepsilon^\frac{p}{p-1}} \mathrm{d}\lambda_{\theta}\\
			&\leq &\int_{\{u_\varepsilon \leq a_\varepsilon / c\}}e^{\mu_\varepsilon u_{\varepsilon,c}^\frac{p}{p-1}} \mathrm{d}\lambda_{\theta}+\frac{c^\frac{p}{p-1}}{a_\varepsilon^{\frac{p}{p-1}}}\int_{\{u_\varepsilon > a_\varepsilon / c\}}e^{\mu_\varepsilon u_\varepsilon^\frac{p}{p-1}} u^\frac{p}{p-1}_\varepsilon \mathrm{d}\lambda_{\theta}\\
			&\leq& \int_0^Re^{\mu_\varepsilon u_{\varepsilon,c}^\frac{p}{p-1}} \mathrm{d}\lambda_{\theta}+\frac{c^\frac{p}{p-1}}{a_\varepsilon^{\frac{p}{p-1}}} \lambda_\varepsilon.
		\end{eqnarray*}
		By using Lemma~\ref{uepsilon} we have $\lim_{\varepsilon \to 0}\|u_{\varepsilon,c}^\prime\|_{L^{p}_\alpha}  \leq c^{-1/p}<1$.  Moreover, the Trudinger-Moser type inequality \eqref{TM1} implies $e^{q \mu_\varepsilon u_{\varepsilon,c}^{p/(p-1)}} \in L^1_\theta$ for some $q>1$. Thus, letting $\varepsilon \to 0$ and $c \to 1$, we obtain the desired result.
		
	\end{proof}
	\noindent By Lemma~\ref{desig}, we obtain 
    \begin{equation}\label{avsL}
        \displaystyle\lim_{\varepsilon \to 0} \frac{a_\varepsilon}{\lambda_\varepsilon} = 0.
    \end{equation}
    In fact, if ${\lambda_\varepsilon}/{a_\varepsilon}$ is bounded,  Lemma~\ref{desig} yields 
	$$
	 S(p,\nu,\theta,R) \leq |B_R|_\theta +\limsup_{\varepsilon \to 0} \frac{\lambda_\varepsilon}{a_\varepsilon} \frac{1}{a_\varepsilon^\frac{1}{p-1}}=|B_R|_{\theta}
	$$
	with is impossible.

We have the following consequence of Lemma~\ref{desig}.
\begin{lemma}\label{coro2} $ S(p,\nu,\theta,R)=|B_R|_{\theta}+
\displaystyle\lim_{r_0\rightarrow\infty}\limsup_{\varepsilon\rightarrow
0}
\int_{0}^{r_{\varepsilon}r_0}e^{\mu_{\varepsilon}u_{\varepsilon}^{\frac{p}{p-1}}}\,\mathrm{d}\lambda_{\theta}.$
\end{lemma}
\begin{proof}
Fixed $r_0>0$, we have
\begin{align*}
	\int_{0}^{r_{\varepsilon}r_0}e^{\mu_{\varepsilon}u_{\varepsilon}^{\frac{p}{p-1}}}\,\mathrm{d}
	\lambda_{\theta}=\frac{\lambda_\varepsilon}{a_{\varepsilon}^{\frac{p}{p-1}}}\int_{0}^{r_0}e^{\mu_{\varepsilon}\big(u^{\frac{p}{p-1}}_\varepsilon(r_\varepsilon s)-a_\varepsilon^{\frac{p}{p-1}}\big)}\,\mathrm{d}
	\lambda_{\theta}.
\end{align*}
From Lemma~\ref{limpsi}, identity \eqref{taylor} and Lemma~\ref{lemmaequal1}, we can write
\begin{align}\label{equilibrio}
\lim_{r_0\to\infty}\limsup_{\varepsilon\to0}	\int_{0}^{r_{\varepsilon}r_0}e^{\mu_{\varepsilon}u_{\varepsilon}^{\frac{p}{p-1}}}\,\mathrm{d}
	\lambda_{\theta}=\limsup_{\varepsilon\to0}\frac{\lambda_\varepsilon}{a_{\varepsilon}^{\frac{p}{p-1}}}.
\end{align}
On the other hand, we have 
$$\int_{r_{\varepsilon}r_0}^{R}e^{\mu_{\varepsilon}u_{\varepsilon}^{\frac{p}{p-1}}}\,\mathrm{d}
\lambda_{\theta}\ge |B_R|_{\theta}-
|B_{r_{\varepsilon}r_0}|_{\theta}$$
which gives 
\begin{equation}\nonumber
	\begin{aligned}
\int_{0}^{r_{\varepsilon}r_0}e^{\mu_{\varepsilon}u_{\varepsilon}^{\frac{p}{p-1}}}\,\mathrm{d}
\lambda_{\theta}&=S_{\varepsilon}(p,\nu,\theta, R)-\int_{r_{\varepsilon}r_0}^{R}e^{\mu_{\varepsilon}u_{\varepsilon}^{\frac{p}{p-1}}}\,\mathrm{d}
\lambda_{\theta}\\
&\le S_{\varepsilon}(p,\nu,\theta, R) -|B_R|_{\theta}+
|B_{r_{\varepsilon}r_0}|_{\theta}.
\end{aligned}
\end{equation}
Since
$\lim_{\varepsilon\rightarrow0}|B_{r_{\varepsilon}r_0}|_{\theta}=0,$
it follows that
\begin{equation}\label{EEE1}
\begin{aligned}
|B_R|_{\theta}  +\limsup_{\varepsilon\rightarrow
0}\int_{0}^{r_{\varepsilon}r_0}e^{\mu_{\varepsilon}|u_{\varepsilon}|^{\frac{p}{p-1}}}\,\mathrm{d}
\lambda_{\theta} &\le S(p,\nu,\theta,R).
\end{aligned}
\end{equation}
From \eqref{equilibrio} and \eqref{EEE1} we obtain 
\begin{align*}
	S(p,\nu,\theta,R)&\ge |B_R|_{\theta}  +\lim_{r_0\to\infty}\limsup_{\varepsilon\rightarrow
		0}\int_{0}^{r_{\varepsilon}r_0}e^{\mu_{\varepsilon}|u_{\varepsilon}|^{\frac{p}{p-1}}}\,\mathrm{d}
	\lambda_{\theta} \\
	&= |B_R|_{\theta}+\limsup_{\varepsilon\to0}\frac{\lambda_\varepsilon}{a_{\varepsilon}^{\frac{p}{p-1}}}.
\end{align*}
Hence, the result follows from  Lemma~\ref{desig}.
\end{proof}
    \begin{lemma}\label{lemadirac} It holds that
\begin{equation*}
\lim_{\varepsilon \to 0}\int_{0}^{R} 
\frac{a_\varepsilon}{\lambda_\varepsilon} 
u_\varepsilon^{\frac{1}{p-1}}
e^{u_\varepsilon^{\frac{p}{p-1}}}
\, v \, \mathrm{d}\lambda_{\theta}
= v(0),
\end{equation*}
for all $v \in C[0,R]$.
\end{lemma}
	\begin{proof}
Let $c>1$.  By Lemma~\ref{to1}, there exists $\varepsilon>0$ sufficiently small such that
$
[0, r_\varepsilon r_0) \subset \left\{ u_\varepsilon > \frac{a_\varepsilon}{c} \right\}.
$
We divide the interval $(0,R]$ into three disjoint parts 
		\begin{equation}
            (0,R]=(0,r_\varepsilon r_0)\cup \big(\{u_\varepsilon> a_\varepsilon/c\}\backslash (0,r_\varepsilon r_0)\big)\cup \{u_\varepsilon \leq a_\varepsilon/c\},
		\end{equation}
		where $\varepsilon > 0$ is small enough. From the change of variables $r = r_\varepsilon s$ we get
		\begin{align*}
			I_1=\int_{0}^{r_{\varepsilon}r_0}\frac{a_\varepsilon}{\lambda_\varepsilon} u_\varepsilon^\frac{1}{p-1}e^{u_\varepsilon^\frac{p}{p-1}} v \mathrm{d}\lambda_{\theta} &= \int_0^{r_0} \varphi_\varepsilon(s)^\frac{1}{p-1}e^{\mu_\varepsilon( u_\varepsilon^{\frac{p}{p-1}}(r_{\varepsilon }s)-a_\varepsilon^\frac{p}{p-1})}v(r_\varepsilon s) \mathrm{d}\lambda_{\theta}\\
			&=v(\tau r_\varepsilon)\int_0^{r_0} \varphi_\varepsilon(s)^\frac{1}{p-1}e^{\mu_\varepsilon( u_\varepsilon^{\frac{p}{p-1}}(r_{\varepsilon}s)-a_\varepsilon^\frac{p}{p-1})}\mathrm{d}\lambda_{\theta},
		\end{align*}
        for some $\tau \in [0, r_0]$.
		Using \eqref{taylor}, Lemma~\ref{lemmaequal1} and letting $\varepsilon \to 0$ and $r_0 \to \infty$, we obtain that $I_1 \to v(0)$ as $\varepsilon \to 0$. In addition,
		\begin{align*}
			I_2&=\int_{\{u_\varepsilon> a_\varepsilon/c\}\backslash (0,r_\varepsilon r_0)}\frac{a_\varepsilon}{\lambda_\varepsilon} u_\varepsilon^\frac{1}{p-1}e^{u_\varepsilon^\frac{p}{p-1}} v \mathrm{d}\lambda_{\theta} \\
            &\leq \|v\|_{L^\infty} \Big(\int_{\{u_\varepsilon> a_\varepsilon/ c\}}\frac{a_\varepsilon}{\lambda_\varepsilon} u_\varepsilon^\frac{1}{p-1}e^{u_\varepsilon^\frac{p}{p-1}} \mathrm{d}\lambda_{\theta}-\int_0^{r_0r_\varepsilon}\frac{a_\varepsilon}{\lambda_\varepsilon} u_\varepsilon^\frac{1}{p-1}e^{u_\varepsilon^\frac{p}{p-1}} \mathrm{d}\lambda_{\theta}\Big)\\
            &\leq\|v\|_{L^\infty}\Big(c\int_{\{u_\varepsilon> a_\varepsilon/c\}}\frac{1}{\lambda_\varepsilon} u_\varepsilon^\frac{p}{p-1}e^{u_\varepsilon^\frac{p}{p-1}} \mathrm{d}\lambda_{\theta} -\int_0^{r_0r_\varepsilon}\frac{a_\varepsilon}{\lambda_\varepsilon} u_\varepsilon^\frac{1}{p-1}e^{u_\varepsilon^\frac{p}{p-1}}\mathrm{d}\lambda_{\theta}\Big)\\
			&\leq\|v\|_{L^\infty}  \Big(c-\int_0^{r_0r_\varepsilon}\frac{a_\varepsilon}{\lambda_\varepsilon} u_\varepsilon^\frac{1}{p-1}e^{u_\varepsilon^\frac{p}{p-1}} \mathrm{d}\lambda_{\theta}\Big).
		\end{align*}
		Letting $\varepsilon \to 0$, $r_0 \to \infty$, and $c \to 1$, using the integral in $I_1$ with $v\equiv 1$ we obtain $I_2\to 0$, as $\varepsilon\to 0$.  Finally, 
		\begin{align}\label{I3-fase 1}
			I_3=\int_{\{u_\varepsilon \leq a_\varepsilon/c\}}\frac{a_\varepsilon}{\lambda_\varepsilon} u_\varepsilon^\frac{1}{p-1}e^{\mu_\varepsilon u_\varepsilon^\frac{p}{p-1}} v \mathrm{d}\lambda_{\theta} &\leq \|v\|_{L^\infty} \frac{a_\varepsilon}{\lambda_\varepsilon}\int_{\{u_\varepsilon \leq a_\varepsilon/c\}} u_\varepsilon^\frac{1}{p-1}e^{\mu_\varepsilon u_\varepsilon^\frac{p}{p-1}} \mathrm{d}\lambda_{\theta}.
		\end{align}
		For $\eta > p$, we can choose $c > 1$ (closely of $1$) and $\varepsilon > 0$ small enough such that
		$$
		\frac{1}{\eta} + \frac{1}{c^{1/(p-1)} - \varepsilon} = 1.
		$$
		Setting $\overline{u}_{\varepsilon,c} = c^{1/p}u_{\varepsilon,c}$, by Lemma~\ref{uepsilon}, we have $\left\|\overline{u}_{\varepsilon,c}^\prime\right\|_{L^{p}_\alpha} \leq 1$.
		Also,
		$$
		\left(c^{\frac{1}{p-1}} - \varepsilon\right) \mu_\varepsilon u_{\varepsilon,c}^{{\frac{p}{p-1}}} = \frac{c^{\frac{1}{p-1}} - \varepsilon}{c^{\frac{1}{p-1}}} \mu_\varepsilon \overline{u}_{\varepsilon,c}^{\frac{p}{p-1}} \leq \mu_\varepsilon \overline{u}_{\varepsilon,c}^{\frac{p}{p-1}}.
		$$
		Since $u_{\varepsilon,c}=u_{\varepsilon}$ on $\{u_{\varepsilon}\le a_{\varepsilon}/c\}$, by H\"{o}lder inequality
		\begin{equation*}
			\begin{aligned}
	\int_{\{u_\varepsilon \leq a_\varepsilon/c\}} u_\varepsilon^\frac{1}{p-1}e^{\mu_\varepsilon u_\varepsilon^\frac{p}{p-1}} \mathrm{d}\lambda_{\theta} &\le 	\int_0^R u_\varepsilon^{\frac{1}{p-1}} e^{\mu_\varepsilon u_{\varepsilon,c}^{\frac{p}{p-1}}}\mathrm{d}\lambda_{\theta}\\
	& \leq \left( \int_0^R u_\varepsilon^{\frac{\eta}{p-1}} \mathrm{d}\lambda_{\theta} \right)^{\frac{1}{\eta}} \left( \int_0^R e^{\mu_\varepsilon \overline{u}_{\varepsilon,c}^{\frac{p}{p-1}}} \mathrm{d}\lambda_\theta \right)^{\frac{1}{c^{1/(p-1)} - \varepsilon}} .
	\end{aligned}
		\end{equation*}
	Thus, from \eqref{TM1},  \eqref{avsL} and \eqref{I3-fase 1},  we have that $ I_3 \to 0 $ as $ \varepsilon \to 0 $.
	\end{proof}
	The following result was proved in  \cite[Lemma~9]{JFcasocritico}.
    \begin{lemma}\label{JFcri}
        Let $g\in X^{1,p}_{R}(\alpha,\theta)$ be a non-increasing function solving the weak equation
        \begin{equation}\nonumber
            \int_{0}^{R}|g^{\prime}|^{p-2}g^{\prime}v\mathrm{d}\lambda_{\alpha}=\int_{0}^{R}fv\mathrm{d}\lambda_{\theta},\;\; \forall\; v\in X^{1,p}_{R}(\alpha,\theta),
        \end{equation}
        where $f\in L^{1}_{\theta}$. Then, for every $0<\chi<\mu_{\alpha,\alpha}/\|f\|^{1/\alpha}_{L^1_{\theta}}$ there holds $e^{\chi g}\in L^{1}_{\alpha}$ and 
        \begin{equation}\nonumber
            \int_{0}^{R}e^{\chi g}\mathrm{d}\lambda_{\alpha}\le C(\alpha, \chi).
        \end{equation}
        Moreover, $g\in X^{1,q}_R(\alpha,\theta)$ for $1<q<p$ and 
        $
            \|g^{\prime}\|_{L^{q}_{\alpha}}\le C(\alpha,q, \|f\|_{L^{1}_{\theta}
            }).$
    \end{lemma}
We apply Lemma~\ref{JFcri} to prove the following:
	\begin{lemma}\label{Lemag2}  
		Let $(f_{\varepsilon})$ be a bounded sequence in $L^{1}_{\theta}$ and $ (g_\varepsilon) \subset X^{1,p}_R(\alpha,\theta)$ with $\theta\ge \alpha$ be a sequence of non-increasing functions satisfying  
		\begin{equation}\label{Eq do lemma} 
			\int_0^R |g^\prime_\varepsilon|^{p-2} g^\prime_\varepsilon v^\prime \, \mathrm{d}\lambda_\alpha = \int_0^R f_\varepsilon v \, \mathrm{d}\lambda_{\theta} + \nu \int_0^R |g_\varepsilon|^{p-2} g_\varepsilon v \, \mathrm{d}\lambda_{\theta}, \quad \forall v \in X^{1,p}_R(\alpha,\theta),  
		\end{equation}  
		where  $ 0 \leq \nu < \lambda_{\alpha,\theta} $. Then, $ g_\varepsilon \in X^{1,q}_R(\alpha,\theta) $ for each $ 1< q < p $ and  
		\begin{equation*}  
			\|g^\prime_\varepsilon\|_{L^{q}_\alpha} \leq C(\alpha, q, c_0),  
		\end{equation*}  
		where $ c_0 $ is an upper bound of $ (f_\varepsilon) $ in $ L^1_\theta $.  
	\end{lemma}  
	\begin{proof}
	      For $\nu=0$, it follows from Lemma~\ref{JFcri}. For $\nu>0$, we claim that
$(g_{\varepsilon})$ is bounded in $L^{p-1}_{\theta}$. Assume by contradiction that
$$
\limsup_{\epsilon\to 0}\|g_{\varepsilon}\|_{L^{p-1}_{\theta}}=\infty
$$ and define
$h_{\varepsilon}=g_{\varepsilon}/\|g_{\varepsilon}\|_{L^{p-1}_{\theta}}$. Then 
$$\|h_{\varepsilon}\|_{L^{p-1}_{\theta}}=1, \;\;\;
h^{\prime}_{\varepsilon}=\frac{g^{\prime}_{\varepsilon}}{\|g_{\varepsilon}\|_{L^{p-1}_{\theta}}}$$
and by \eqref{Eq do lemma}
\begin{equation}\label{E13}
\int_{0}^{R}|h^{\prime}_{\varepsilon}|^{p-2}h^{\prime}_{\varepsilon}v^{\prime}\;\mathrm{d}\lambda_{\alpha}=\int_{0}^{R}\overline{f}_{\varepsilon}v\;\mathrm{d}\lambda_{\theta},\;\;\mbox{with}\;\; \overline{f}_{\varepsilon}=\frac{1}{\|g_{\varepsilon}\|^{p-1}_{L^{p-1}_{\theta}}}\left(f_{\varepsilon}+\nu
|g_{\varepsilon}|^{p-1}\right),
\end{equation}
for all $v \in X^{1,p}_R(\alpha,\theta)$.
Note that $(\overline{f}_{\varepsilon})$ is bounded in
$L^{1}_{\theta}$. Then, from Lemma~\ref{JFcri}  we conclude that  $\|h^{\prime}_{\varepsilon}\|_{L^{s}_{\alpha}}\le c_0$, for $1<s<p$. Since we are assuming $\theta\ge \alpha>\alpha-s$, by \eqref{cond-equivNorm} we get  $(h_{\varepsilon})$ bounded in $X^{1,s}_R(\alpha,\theta)$. In particular, up to a subsequence, $h_{\varepsilon}\rightharpoonup h$ in $X^{1,s}_R(\alpha,\theta)$, for  $1<s<p$.   Since $\theta\ge\alpha=p-1$, for  $p/2<s<p$, we get
\begin{align*}
\alpha-s+1> \alpha-p+1=0 \;\; \mbox{and}\;\; s^{*}(\alpha,s,\theta)=\frac{(\theta+1)s}{\alpha-s+1}\ge \frac{ps}{p-s}>p.   
\end{align*}
Thus, the compact embedding \eqref{imersão-Sobolev} yields   $h_{\varepsilon}\rightarrow h$ in $L^{p-1}_{\theta}$ and $ h_{\varepsilon}(r)\rightarrow h(r)$ a.e. in $(0,R)$. It follows that 
\begin{equation}\label{normah=1}
    \|h\|_{L^{p-1}_{\theta}}=\lim_{\varepsilon  \rightarrow
0}\|h_{\varepsilon}\|_{L^{p-1}_{\theta}}=1.
\end{equation}
However, the
equation \eqref{E13} implies that  $h$ satisfies
$$
\int_{0}^{R}|h^{\prime}|^{p-2}h^{\prime}v^{\prime}\;\mathrm{d}\lambda_{\alpha}=\nu\int_{0}^{R}|h|^{p-2}hv\;\mathrm{d}\lambda_{\theta},\;\; \forall\; v \in X^{1,p}_R(\alpha,\theta).
$$
Since $\nu<\lambda_{\alpha,\theta}$, we have $h\equiv 0$, which contradicts \eqref{normah=1}. Thus, $(g_{\varepsilon})$ is bounded in $L^{p-1}_{\theta}$ and equation \eqref{Eq do lemma} can be written as
$$
\int_{0}^{R}|g^{\prime}_{\varepsilon}|^{p-2}g^{\prime}_{\varepsilon}v^{\prime}\;\mathrm{d}\lambda_{\alpha}=
\int_{0}^{R}\tilde{f}_{\varepsilon}v\;\mathrm{d}\lambda_{\theta}, \;\; \forall\; v \in X^{1,p}_R(\alpha,\theta),
$$
where $\tilde{f}_{\varepsilon}=f_{\varepsilon}+\nu|g_{\varepsilon}|^{p-2}g_{\varepsilon}$ is bounded in $L^{1}_{\theta}$. The conclusion now follows from Lemma~\ref{JFcri} with $f=\tilde{f}_{\varepsilon}$.
	\end{proof}
	\begin{lemma}\label{lemag}  
		Let $ p \geq 2,$ $\theta\ge \alpha$  and $ p/2< q < p $ and let $ 0 \leq \nu < \lambda_{\alpha,\theta} $. Then there exists a function $ g=g_{\nu} $ such that $ a_\varepsilon^{\frac{1}{p-1}} u_\varepsilon \rightharpoonup g $ in $ X^{1,q}_R(\alpha,\theta) $ and  
		\begin{equation}\label{eqG}  
			\int_0^R |g^\prime|^{p-2} g^\prime v^\prime \, \mathrm{d}\lambda_\alpha = \delta_0(v) + \nu \int_0^R |g|^{p-2} g v \, \mathrm{d}\lambda_{\theta}, \quad \forall\; v \in X^{1,p}_R(\alpha,\theta) \cap C[0,R], 
		\end{equation}  
       where $\delta_0$  denotes the Dirac measure concentrated at the origin. In addition,
		\begin{enumerate}  
			\item[$i)$] $ a_\varepsilon^{\frac{1}{p-1}} u_\varepsilon \to g $ in $ C^0_{\text{loc}}[0,R] $;
			\item[$ii)$] $ a_\varepsilon^{\frac{1}{p-1}} u^\prime_\varepsilon \to g^\prime $ in $ L_\alpha^{p}(r_1, R] $ for all $ r_1 > 0 $; 
			\item[$iii)$]  $g$ has the form
			\begin{equation}\label{Definig}
			g(r) = -\frac{\theta+1}{\mu_{\alpha,\theta}} \ln r + A_0 + z(r),  
			\end{equation} 
			where $A_0$ is a constant and $z \in C^1[0,R] $ and $ z^{\prime}(0) = z(0)=0$.  
		\end{enumerate}  
	\end{lemma}
	\begin{proof}
	From \eqref{Lagrange} it follows that  
		\begin{align*}  
			\int_0^R |a_\varepsilon^{\frac{1}{p-1}} u^\prime_\varepsilon|^{p-2} a_\varepsilon^\frac{1}{p-1} u_\varepsilon^\prime v^\prime \, \mathrm{d}\lambda_\alpha = \frac{a_{\varepsilon}}{\lambda_\varepsilon}\int_0^R  e^{\mu_\varepsilon u_\varepsilon^\frac{p}{p-1}} u_\varepsilon^\frac{1}{p-1} v \, \mathrm{d}\lambda_{\theta} + \nu \int_0^R |a_\varepsilon^\frac{1}{p-1} u_\varepsilon|^{p-2} a_\varepsilon^\frac{1}{p-1} u_\varepsilon v \, \mathrm{d}\lambda_\theta.  
		\end{align*}  
By Lemma~\ref{lemadirac}, we know that $$ \frac{a_{\varepsilon}}{\lambda_\varepsilon} e^{\mu_\varepsilon |u_\varepsilon|^{\frac{p}{p-1}}}u_\varepsilon^\frac{1}{p-1}\;\;\mbox{is bounded in}\;\;  L^1_\theta.$$  So, Lemma~\ref{Lemag2} yields $ \|a_\varepsilon^\frac{1}{p-1} u^\prime_\varepsilon\|_{L^{q}_\alpha} \leq c $. Hence, $a_\varepsilon^\frac{1}{p-1} u_\varepsilon \rightharpoonup g $ in $X^{1,q}_R(\alpha,\theta)$.  Our assumptions on $\alpha, p,q$ and $\theta$ ensure that  $\alpha-q+1>0$ and 
$$
q^*=\frac{(\theta+1)q}{\alpha-q+1}>p.
$$
Thus,   the compact embedding \eqref{imersão-Sobolev} implies  
		\begin{equation}\label{conver}  
			a_\varepsilon^\frac{1}{p-1} u_\varepsilon \to g \text{ in } L^{p}_\theta\; \text{ and } \; a_\varepsilon^\frac{1}{p-1} u_\varepsilon(r) \to g(r) \text{ a.e. on } (0,R).  
		\end{equation}  
		Arguing as in \eqref{defuprime}, we have  
		\begin{eqnarray}\label{lemau}  
			\omega_\alpha |a_\varepsilon^\frac{1}{p-1} u^\prime_\varepsilon|^{p-1} = \frac{1}{r^\alpha} \int_0^r \frac{a_{\varepsilon}}{\lambda_\varepsilon} u_\varepsilon^\frac{1}{p-1} e^{\mu_\varepsilon u_\varepsilon^\frac{p}{p-1}} \, \mathrm{d}\lambda_{\theta} + \frac{\nu}{r^\alpha} \int_0^r (a_\varepsilon^\frac{1}{p-1} u_\varepsilon)^{p-2} a_\varepsilon^\frac{1}{p-1} u_\varepsilon \, \mathrm{d}\lambda_{\theta}
		\end{eqnarray}  
		and then 
		\begin{eqnarray}\label{eqg}
			a_\varepsilon^\frac{1}{p-1}u_\varepsilon=\frac{1}{\omega^{\frac{1}{p-1}}_{\alpha}}\int_r^R\frac{1}{t}\left(\int_0^t \Big[\frac{a_{\varepsilon}}{\lambda_\varepsilon}u_\varepsilon^\frac{1}{p-1}e^{\mu_\varepsilon u_\varepsilon^\frac{p}{p-1}}+\nu|a_\varepsilon^\frac{1}{p-1}u_\varepsilon|^{p-2} a_\varepsilon^\frac{1}{p-1}u_\varepsilon\Big]\mathrm{d}\lambda_{\theta}\right)^\frac{1}{p-1}\mathrm{d} t.
		\end{eqnarray}
		 By using  Lemma~\ref{lemadirac} and \eqref{conver}, taking $\varepsilon \to 0 $  we obtain
		\begin{equation}\label{gequation}
			g(r)=\frac{\theta+1}{\mu_{\alpha,\theta}}\int_r^R \frac{1}{t}\left(1+\nu\int_0^t |g|^{p-1}\mathrm{d}\lambda_{\theta}\right)^\frac{1}{p-1}\mathrm{d} t.
		\end{equation}
Hence, 
\begin{equation}\label{glinha}  
	-g^\prime(r) =\frac{1}{\omega^{\frac{1}{p-1}}_{\alpha}}\frac{1}{r}\left(1 + \nu \int_0^r |g|^{p-1} \mathrm{d}\lambda_{\theta} \right)^\frac{1}{p-1}  
\end{equation}  
which yields 
\begin{equation}\label{derivg}  
	-\omega_\alpha(r^\alpha |g^\prime|^{p-2} g^\prime) = 1 + \nu \int_0^r |g|^{p-1} \mathrm{d}\lambda_{\theta}.  
\end{equation}  
For each $ v \in X^{1,p}_R(\alpha,\theta) \cap C[0,R] $, multiplying \eqref{derivg} by $ v^\prime $ and integrating over $ (0,R) $, we can see that $ g $ satisfies \eqref{eqG}.

		$i)$ Let  $r_1\in (0,R)$ be fixed. From \eqref{conver}, we have that $ a_\varepsilon^\frac{1}{p-1}u_\varepsilon $ is bounded in $ L^{p}_\theta$. Thus,  combining \eqref{lemau} with Lemma~\ref{lemadirac}, we obtain $$|a_\varepsilon^\frac{1}{p-1}u^\prime_\varepsilon(r)| \leq \frac{C_2}{r_1}\;\;\mbox{in }  [r_1,R], $$ where $ C_2 $ depends on $ p, \nu, \theta $. Similarly, \eqref{eqg} shows that $ a_\varepsilon^\frac{1}{p-1}u_\varepsilon $ is bounded in $ C[r_1,R] $. Thus, the Arzelà-Ascoli theorem implies that $ a_\varepsilon^\frac{1}{p-1} u_\varepsilon $ converges to $ \Tilde{g} $ in $ C[r_1,R] $, and by \eqref{lemau}, we conclude that $ g = \Tilde{g} $.

		$ii)$ As in the previous item, we have $ |a_\varepsilon^\frac{1}{p-1}u^\prime_\varepsilon(r)| \leq \frac{C_2}{r_1} $ for all $ r \in [r_1,R] $. Moreover, combining  \eqref{lemau} and \eqref{derivg}, it follows that $ a_\varepsilon^\frac{1}{p-1}u^\prime_\varepsilon(r) \to g^\prime(r) $ almost everywhere in $ [r_1,R] $. By the Lebesgue dominated convergence theorem, the result follows.

$iii)$  Let 
\[
\sigma(t)
=
\frac{1}{t}
\left[
\left(1+\nu\int_0^t |g|^{p-1}\,\mathrm d\lambda_\theta\right)^{\frac{1}{p-1}}
-1
\right],
\qquad t\in(0,R].
\]
By using \eqref{gequation}, we can write
\begin{equation}\label{gequationVar}
\begin{aligned}
			g(r)&=\frac{\theta+1}{\mu_{\alpha,\theta}}\int_r^R\Big(\frac{1}{t}+\sigma(t)\Big)\mathrm{d} t=-\frac{\theta+1}{\mu_{\alpha,\theta}}(\ln r-\ln R )+\frac{\theta+1}{\mu_{\alpha,\theta}}\int_r^R\sigma(t)\mathrm{d} t.
\end{aligned}
		\end{equation}
By Hölder's inequality 
$$
\int_0^t |g|^{p-1}\mathrm{d}\lambda_{\theta}
\le
\left(\int_0^t |g|^p\mathrm{d}\lambda_{\theta}\right)^{\frac{p-1}{p}}
\left(\int_0^t\mathrm{d}\lambda_{\theta}\right)^{\frac{1}{p}}=\frac{\omega_{\theta}}{\theta+1}\left(\int_0^t |g|^p\mathrm{d}\lambda_{\theta}\right)^{\frac{p-1}{p}} t^{\frac{\theta+1}{p}}.
$$
Now, $g\in L^{p}_{\theta}$  yields $\int_{0}^{t}|g|^{p}\mathrm{d}\lambda_{\theta}\to 0$ as $t\to 0$. Thus, since $ \theta+1 \geq p $, we have 
\begin{equation}\label{order-g1}
    \int_0^t |g|^{p-1}\mathrm{d}\lambda_{\theta}=o(t),\;\; \mbox{as}\;\; t \to 0.
\end{equation}
 By mean value theorem and \eqref{order-g1}, for some $0<\xi_t<1$
\begin{align*}
  \left(1+\nu\int_0^t|g|^{p-1}\mathrm{d}\lambda_{\theta}\right)^\frac{1}{p-1}-1&=\frac{\nu}{p-1}\Big(1+\nu\xi_{t}\int_0^t|g|^{p-1}\mathrm{d}\lambda_{\theta}\Big)^{\frac{2-p}{p-1}}\int_0^t|g|^{p-1}\mathrm{d}\lambda_{\theta}=o(t), 
\end{align*}
as $t\to 0$. It follows  that 
\begin{equation}\label{Vartheta}
  \lim_{t\to 0}\sigma(t)= 0.
\end{equation}
Hence, from \eqref{gequationVar} and \eqref{Vartheta} we obtain
\begin{equation}\label{gequationVAR}
\begin{aligned}
			g(r)=-\frac{\theta+1}{\mu_{\alpha,\theta}}\ln r+ A_0+ z(r)
\end{aligned}
		\end{equation}
        where 
        \begin{equation}\label{Az}
            A_0=\frac{\theta+1}{\mu_{\alpha,\theta}}\Big(\int_{0}^{R}\sigma(t)\mathrm{d} t+\ln R\Big)\;\;\mbox{and}\;\; z(r)=-\frac{\theta+1}{\mu_{\alpha,\theta}}\int_{0}^{r}\sigma(t)\mathrm{d} t. 
        \end{equation}
From \eqref{Vartheta}, we have $ z \in C^1[0,R] \cap C[0,R] $. Further, from L'Hospital rule
\begin{equation}
   \lim_{r\to 0}\frac{z(r)}{r}=-\frac{\theta+1}{\mu_{\alpha,\theta}}\lim_{r\to 0}\sigma(r)=0.
\end{equation}  
This, proves  \textit{iii)}.
\end{proof}
\begin{lemma}\label{lemadaintegral}
Let $a>0$ and $x>1$ we have
\begin{equation}\label{integralema}
\int_0^a\frac{s^{x-1}}{(1+s)^{x}} \mathrm{d} s=\ln (1+a)-[\Psi(x)+\gamma]+O\Big(\frac{1}{a}\Big),\;\; \mbox{as}\;\; a\to \infty.
\end{equation}
where $\Psi(x)=\Gamma^\prime(x)/\Gamma(x)$ and $\gamma=\lim_{m \to \infty}(\sum_{j=1}^m \frac{1}{j}-\ln m)$.
\end{lemma}
\begin{proof}
From \cite[Lemma~19]{CV}, we can write
\begin{equation*}
    \int_0^a\frac{s^{x-1}}{(1+s)^{x}} \mathrm{d} s=\ln (1+a)-\Psi(x)-\gamma+\int_{\frac{a}{a+1}}^{1}\frac{1-s^{x-1}}{1-s}\mathrm{d} s. 
\end{equation*}
 By applying L'Hospital rule,  we obtain
 \begin{align*}
     \lim_{a\to\infty}a\int_{\frac{a}{a+1}}^{1}\frac{1-s^{x-1}}{1-s}\mathrm{d} s=x-1,
 \end{align*}
which completes the proof.
\end{proof}
\noindent For $0<a<b$, let us define the weighted Sobolev space
$$
W^{1,p}_{(a,b)}(\alpha,\theta)=\Big\{u\in AC_{loc}(a,b)\;:\; u\in L^{p}_{\theta}(a,b)\;\;\mbox{and}\;\; u^{\prime}\in L^{p}_{\alpha}(a,b)\Big\}
$$
endowed with the norm 
$$
\|u\|_{W^{1,p}}=\big(\|u\|^{p}_{L^{p}_{\theta}(a,b)}+\|u^{\prime}\|^{p}_{L^{p}_{\alpha}(a,b)}\big)^{\frac{1}{p}}.
$$
We note that $W^{1,p}_{(a,b)}(\alpha,\theta)$ is reflexive whenever $p>1$. In addition, we have the following:
\begin{lemma}\label{Lemmafronteira-W}  Let $u \in W^{1,p}_{(a,b)}(\alpha,\theta)$. Then there exist positive constants 
$C,\bar C>0$, independent of $u$, such that  
    \begin{equation}\label{fronteira-W}
    \max\{|u(a)|, |u(b)|\}\le C\|u\|_{W^{1,p}}
\end{equation}
and 
\begin{equation}\label{Poincaré}
    \|u-u(a)\|_{L^{p}_{\theta}(a,b)}\le \bar{C}\|u^{\prime}\|_{L^{p}_{\alpha}}.
\end{equation}
\end{lemma}
\begin{proof}
Fix $u\in W^{1,p}_{(a,b)}(\alpha,\theta)$. By  mean value theorem for  integrals, there is $\xi\in\mathbb{R}$ with $(a+b)/2<\xi\le b$ such that
\begin{equation}\label{estimativa-medio}
\begin{aligned}
    |u(\xi)|& =\Big|\frac{2}{b-a}\int_{\frac{a+b}{2}}^{b}u(s)\mathrm{d} s\Big|\\
    &\le \frac{2}{b-a}\|u\|_{L^p_{\theta}(a,b)}\Big(\omega^{-\frac{1}{p-1}}_{\theta}\int_{\frac{a+b}{2}}^{b}s^{-\frac{\theta}{p-1}}\mathrm{d} s\Big)^{\frac{p-1}{p}}\\
    &\le c\|u\|_{L^p_{\theta}(a,b)},
    \end{aligned}
\end{equation}
where we also have used the H\"{o}lder inequality. From \eqref{estimativa-medio}, we can write
\begin{equation}\label{estimativa-b}
\begin{aligned}
   |u(b)|& \le |u(\xi)|+|u(b)-u(\xi)|\\
   &=c\|u\|_{L^p_{\theta}(a,b)}+\Big|\int_{\xi}^{b}u^{\prime}(s)\mathrm{d} s\Big|\\
    &\le c\|u\|_{L^p_{\theta}(a,b)}+ \|u^{\prime}\|_{L^p_{\alpha}(a,b)}\Big(\omega^{-\frac{1}{p-1}}_{\alpha}\int_{\frac{a+b}{2}}^{b}s^{-\frac{\alpha}{p-1}}\mathrm{d} s\Big)^{\frac{p-1}{p}}\\
    &\le C_1\|u\|_{W^{1,p}},
    \end{aligned}
\end{equation}
for some $C_1>0$ depending only on $a,b, \alpha, \theta $ and $p$. Analogously, we can see that 
\begin{equation}\label{estimativa-a}
\begin{aligned}
   |u(a)|  &\le C_2\|u\|_{W^{1,p}},
    \end{aligned}
\end{equation}
for  some $C_2>0$ independent of $u$.  From \eqref{estimativa-b} and \eqref{estimativa-a}, we get \eqref{fronteira-W}.  To obtain \eqref{Poincaré}, we note that 
\begin{align*}
    |u-u(a)|\le \int_{a}^{r}|u^{\prime}(s)|ds\le \|u^{\prime}\|_{L^{p}_{\alpha}}\Big(\omega^{-\frac{1}{p-1}}_{\alpha}\int_{a}^{b}s^{-\frac{\alpha}{p-1}}\Big)^{\frac{p-1}{p}}
\end{align*}
and then 
\begin{align*}
    \int_{a}^{b}|u-u(a)|^{p}\mathrm{d}\lambda_{\theta}\le \|u^{\prime}\|^p_{L^{p}_{\alpha}}\Big(\omega^{-\frac{1}{p-1}}_{\alpha}\int_{a}^{b}s^{-\frac{\alpha}{p-1}}\Big)^{p-1}\int_{a}^{b}\mathrm{d}\lambda_{\theta}\le \tilde{C}\|u^{\prime}\|^p_{L^{p}_{\alpha}}
\end{align*}
for some $\tilde{C}>0$ depending only on $a,b,\alpha,\theta$ and $p$. 
\end{proof}
	\begin{lemma}\label{lemacontrad}
	Suppose $ 0 < \nu < \lambda_{\alpha,\theta} $. Under the condition \eqref{Ablow-up}, it holds that  
\begin{equation}\label{teodesig}
	S(p,\nu,\theta,R)\leq |B_R|_\theta + e^{\mu_{\alpha,\theta}A_{0}+ \Psi(p)+\gamma}|B_1|_\theta,
\end{equation}
	where $A_0$ is defined in Lemma~\ref{lemag} and $\Psi$ and $\gamma$ are given in Lemma~\ref{lemadaintegral}.
\end{lemma}

\begin{proof}
Let $0<\delta<1/2$. For  fixed $L>0$, we have $[0,r_\varepsilon L) \subset [0, \delta)$, if $\varepsilon>0$ small enough. Since $u_\varepsilon$ is decreasing, we have $u_\varepsilon(r_\varepsilon L)>u_\varepsilon(\delta)$. Note that 
$$
\Big\{u_{|[r_{\varepsilon}L, \delta]}\;:\; u\in X^{1,p}_R(\alpha,\theta)\Big\}\subset W^{1,p}_{(r_{\varepsilon}L, \delta)}(\alpha, \theta).
$$
Define 
$$
K_{\varepsilon}=\{u \in W^{1,p}_{(r_{\varepsilon}L, \delta)}(\alpha, \theta)\;:\;\ u(r_\varepsilon L)=u_\varepsilon(r_\varepsilon L)\;\mbox{and}\;\; u(\delta)=u_\varepsilon(\delta)\}
$$
and 
 $J: K_{\varepsilon }\to \mathbb{R}$ be given by 
$$
J(u)=\int_{r_\varepsilon L}^\delta |u^\prime|^{p}\mathrm{d}\lambda_\alpha.
$$
By \eqref{fronteira-W}, $K_{\epsilon}$ is a closed, convex subset of $W^{1,p}_{(r_{\varepsilon}L, \delta)}(\alpha, \theta)$. Noticing that 
 $J$ is convex and continuous, we have that it is weakly lower-semicontinuous. In addition, let $(v_n)\subset K_{\varepsilon}$ such that
 \begin{equation}\label{j-Wnorma}
 J(v_n)+ \|v_n\|^{p}_{L^{p}_{\theta}(r_{\varepsilon L,\delta})}=\|v_n\|^{p}_{W^{1,p}}\to+\infty,\;\;\mbox{as} \;\;n\to\infty.
 \end{equation} 
 From \eqref{Poincaré}, for any $n\in\mathbb{N}$ we can write   
 \begin{align*}
     \|v_n\|_{L^{p}_{\theta}(r_{\varepsilon L,\delta})}&\le  \|v_n-v_n(r_{\varepsilon} L)\|_{L^{p}_{\theta}(r_{\varepsilon L,\delta})}+ \|v_n(r_{\varepsilon} L)\|_{L^{p}_{\theta}(r_{\varepsilon L,\delta})}\\
     &\le \bar{C}\|v^{\prime}_{n}\|_{L^{p}_{\alpha}(r_{\varepsilon} L, \delta)}+|v_n(r_{\varepsilon} L)|\Big(\int_{r_{\varepsilon} L}^{\delta}\mathrm{d}\lambda_{\theta}\Big)^{\frac{1}{p}}\\
     & = \bar{C}[J(v_n)]^{\frac{1}{p}}+|u_{\epsilon}(r_{\varepsilon L})|\Big(\int_{r_{\varepsilon} L}^{\delta}\mathrm{d}\lambda_{\theta}\Big)^{\frac{1}{p}}.
 \end{align*}
 Therefore, from \eqref{j-Wnorma} we have $J(v_n)\to +\infty$, as $n\to\infty$. Thus, $J$ is coercive. Hence, $J$ admits  a minimizer $h \in K_{\varepsilon}$ (see for instance \cite[Proposition~5.1.1]{Kesavan}) which satisfies
\begin{equation}\label{EDO-h}
    \left\{\begin{aligned}
    &(r^\alpha|h^\prime|^{p-2} h^\prime)^\prime=0 \quad \text{in } (r_\varepsilon L, \delta)  \\
    & h(r_\varepsilon L)=u_\varepsilon(r_\varepsilon L)\;\;\mbox{and} \;\;h(\delta)=u_\varepsilon(\delta).
    \end{aligned}\right. 
\end{equation}
The explicit solution of \eqref{EDO-h} is given by
$$
h(r)=\frac{u_\varepsilon(\delta)(\ln r-\ln(r_\varepsilon L))+u_\varepsilon(r_\varepsilon L)(\ln \delta-\ln r)}{\ln \delta - \ln(r_\varepsilon L)}.
$$
Consequently, we have
\begin{equation}\label{min-value}
\mathcal{m}=\min_{K_{\varepsilon}}J=\int_{r_\varepsilon L}^\delta |h^\prime|^{p} \mathrm{d}\lambda_\alpha=\omega_\alpha\frac{(u_\varepsilon(r_\varepsilon L)-u_\varepsilon(\delta))^{p}}{(\ln \delta-\ln (r_\varepsilon L))^{p-1}}.
\end{equation}
Next, we derive upper and lower estimates for the quantity
\begin{equation}\label{Min-value}
\mathcal{m}^{\frac{1}{p-1}}=\frac{\mu_{\alpha,\theta}}{\theta+1}\frac{(u_\varepsilon(r_\varepsilon L)-u_\varepsilon(\delta))^{\frac{p}{p-1}}}{\ln \delta-\ln (r_\varepsilon L)}.
\end{equation}

Let $\overline{u}_\varepsilon=\max\{u_\varepsilon(\delta),\min\{u_\varepsilon, u_\varepsilon(r_\varepsilon L)\}\}$, we have ${\overline{u}_\varepsilon}_{|_{[r_{\varepsilon}L, \delta]}} \in K_\varepsilon$. Thus, 
\begin{equation}\label{derivh}
    \begin{aligned}
        \int_{r_\varepsilon L}^\delta |h^\prime|^{p} \mathrm{d}\lambda_\alpha&\leq \int_{r_\varepsilon L}^\delta |\overline{u}_\varepsilon^\prime|^{p} \mathrm{d}\lambda_\alpha\\
        &\leq \int_{r_\varepsilon L}^\delta |u_\varepsilon^\prime|^{p} \mathrm{d}\lambda_\alpha\\
        &=1+\nu\|u_\varepsilon\|^{p}_{L^{p}_\theta}-\int_0^{r_\varepsilon L} |u^\prime_\varepsilon|^{p} \mathrm{d}\lambda_\alpha-\int_\delta^{R} |u^\prime_\varepsilon|^{p} \mathrm{d}\lambda_\alpha,
    \end{aligned}
\end{equation}
where we have used that $H_{\nu}(u_{\varepsilon})=1$.
By $ii)$ in Lemma~\ref{lemag} we have
\begin{equation}\label{res1}
    \int_\delta^{R} |u^\prime_\varepsilon|^{p} \mathrm{d}\lambda_\alpha= \frac{1}{a_\varepsilon^{\frac{p}{p-1}}}\int_\delta^R |g^\prime|^{p}\mathrm{d}\lambda_\alpha+o_\varepsilon(1).
\end{equation}
Note that \eqref{gequationVar} and \eqref{Az}  ensure $g(R)=0$. By integrating by parts in \eqref{derivg} and using \eqref{order-g1}, we obtain
\begin{equation}\label{equationg}
    \begin{aligned}
\int_\delta^R |g^\prime|^{p}\mathrm{d}\lambda_\alpha&=-\int_\delta^R(1+\nu\int_0^r|g|^{p-1}\mathrm{d} \lambda_\theta)g^\prime \mathrm{d} r\\
&=g(\delta)+\nu g(\delta) \int_0^\delta |g|^{p-1} \mathrm{d} \lambda_\theta+\nu \int_\delta^R |g|^{p} \mathrm{d} \lambda_\theta\\
&=-\frac{\theta+1}{\mu_{\alpha,\theta}}\ln \delta+A_0+z(\delta)+\nu \int_\delta^R|g|^p \mathrm{d}\lambda_\theta+O(\delta\ln \delta),
\end{aligned}    
\end{equation}
as $\delta \to 0$. Note that 
 $$\int_\delta^R|g|^p \mathrm{d}\lambda_\theta=\int_0^R|g|^p \mathrm{d}\lambda_\theta-\int_0^\delta|g|^p \mathrm{d}\lambda_\theta=\|g\|^{p}_{L^p_{\theta}}+O(\delta^{\theta}|\ln \delta|^{p}).$$ Consequently,  since we have $z(\delta)=o(\delta)$, as $\delta\to 0$
\begin{equation}\label{ulinhadeu}
\begin{aligned}
    \int_\delta^R |u^\prime_\varepsilon|^{p}\mathrm{d} \lambda_\alpha &= \frac{1}{a_\varepsilon^{\frac{p}{p-1}}}\Big[-\frac{\theta+1}{\mu_{\alpha,\theta}}\ln \delta+A_0+\nu \|g\|^{p}_{L^{p}_\theta}+O(\delta^{\theta}|\ln \delta|^{p})+O(\delta\ln \delta)+o(\delta)\Big]\\
    &=\frac{1}{a_\varepsilon^{\frac{p}{p-1}}}\Big[-\frac{\theta+1}{\mu_{\alpha,\theta}}\ln \delta+A_0+\nu \|g\|^{p}_{L^{p}_\theta}+O(\delta^{\theta}|\ln \delta|^{p})+O(\delta\ln \delta)\Big].
    \end{aligned}
\end{equation}
By \eqref{phipsi} and Lemma~\ref{limpsi}
\begin{equation}\label{ulinha}
    \begin{aligned}
        \int_0^{r_\varepsilon L} |u^\prime_\varepsilon|^{p}\mathrm{d}\lambda_{\alpha}&= \frac{1}{a_\varepsilon^{\frac{p}{p-1}}}\int_0^{L}|\psi^\prime_\varepsilon|^{p}\mathrm{d} \lambda_\alpha\\
        &= \frac{1}{a_\varepsilon^{\frac{p}{p-1}}}\Big[\int_0^{L}|\psi^\prime|^{p}\mathrm{d} \lambda_\alpha+o_\varepsilon(1)\Big].
\end{aligned}
\end{equation}
Observe that 
\begin{equation}\label{integralpsi}
\int_0^{L}|\psi^\prime|^{p}\mathrm{d}\lambda_\alpha=\frac{p-1}{\mu_{\alpha,\theta}}\int_0^{c_{0}L^{\frac{\theta+1}{p-1}}}\frac{s^{p-1}}{(1+s)^{p}}\mathrm{d} s
\end{equation}
and by Lemma~\ref{lemadaintegral} 
\begin{equation}\label{lemapsilinha}
    \begin{aligned}
        \int_0^L|\psi^\prime|^{p}\mathrm{d} \lambda_\alpha=& \frac{p-1}{\mu_{\alpha,\theta}}\left[\ln \Big(1+c_{0}L^{\frac{\theta+1}{p-1}}\Big)-\Psi(p)-\gamma+O\Big(L^{-\frac{\theta+1}{p-1}}\Big)\right]\\
        =& \frac{p-1}{\mu_{\alpha,\theta}}\left[\frac{\theta+1}{p-1}\ln L+\ln c_{0}-\Psi(p)-\gamma+o_L(1)\right],
    \end{aligned}
\end{equation}
where $o_L(1)\to 0$, as $L\to+\infty$. From  \eqref{ulinha} and \eqref{lemapsilinha}, we obtain
\begin{equation}\label{Derivdau}
\begin{aligned}
    \int_0^{r_\varepsilon L}|u^\prime_\varepsilon|^{p}\mathrm{d} \lambda_\alpha&=\frac{1}{a_\varepsilon^{\frac{p}{p-1}}}\Big[\frac{p-1}{\mu_{\alpha,\theta}}\Big(\frac{\theta+1}{p-1}\ln  
 L+\ln c_{0}-\Psi(p)-\gamma+o_L(1)\Big)+o_\varepsilon(1)\Big]\\
 &=\frac{1}{a_\varepsilon^{\frac{p}{p-1}}}\Big[\frac{\theta+1}{\mu_{\alpha,\theta}}\ln   L+\frac{1}{\mu_{\alpha,\theta}}\ln \frac{\omega_{\theta}}{\theta+1}-\frac{p-1}{\mu_{\alpha,\theta}}[\Psi(p)+\gamma]+o_L(1)+o_\varepsilon(1)\Big],
 \end{aligned}
\end{equation}
since $c_0=\Big(\frac{\omega_{\theta}}{\theta+1}\Big)^{\frac{1}{p-1}}$.
By \eqref{conver}, we have
\begin{equation}\label{assintotic}
\|u_\varepsilon\|^{p}_{L^{p}_\theta}=\frac{1}{a_\varepsilon^{\frac{p}{p-1}}}\Big(\|g\|^{p}_{L^{p}_\theta}+o_\varepsilon(1)\Big).
\end{equation}
By \eqref{derivh},  \eqref{equationg}, \eqref{Derivdau} and \eqref{assintotic}, we obtain
\begin{equation}\label{des1}
\begin{aligned}
    \int_{r_\varepsilon L}^\delta |h^\prime|^{p}\mathrm{d}\lambda_\alpha &\le 1+\frac{1}{a_\varepsilon^{\frac{p}{p-1}}}\Big(\nu\|g\|^{p}_{L^{p}_\theta}+o_\varepsilon(1)\Big)\\
    &-\frac{1}{a_\varepsilon^{\frac{p}{p-1}}}\Big[\frac{\theta+1}{\mu_{\alpha,\theta}}\ln   L+\frac{1}{\mu_{\alpha,\theta}}\ln \frac{\omega_{\theta}}{\theta+1}-\frac{p-1}{\mu_{\alpha,\theta}}[\Psi(p)+\gamma]+o_L(1)+o_\varepsilon(1)\Big]\\
    &-\frac{1}{a_\varepsilon^{\frac{p}{p-1}}}\Big[-\frac{\theta+1}{\mu_{\alpha,\theta}}\ln \delta+A_0+\nu \|g\|^{p}_{L^{p}_\theta}+O(\delta^{\theta}|\ln \delta|^{p})+O(\delta\ln \delta)\Big]\\
    &=1+\frac{1}{a_\varepsilon^{\frac{p}{p-1}}}\Big[\frac{\theta+1}{\mu_{\alpha,\theta}}\ln \frac{\delta}{L}-A_0-\frac{1}{\mu_{\alpha,\theta}}\ln \frac{\omega_\theta}{\theta+1}+\frac{p-1}{\mu_{\alpha,\theta}}[\Psi(p)+\gamma]\\
    &\quad\quad\quad\quad\quad+O(\delta^{\theta}|\ln \delta|^{p})+O(\delta\ln \delta)+o_L(1)+o_\varepsilon(1)\Big].
\end{aligned}
\end{equation}
Note that the right-hand side of the above estimate satisfies $0 \le \text{RHS} \le 1$ for sufficiently small $\varepsilon, \delta$ and sufficiently large $L$.
We recall that $(1-t)^a \leq 1-at$ for $0\le t\le 1$,  if $0<a<1$. Thus, \eqref{min-value} and \eqref{des1} imply
\begin{equation}\label{desbaixo}
    \begin{aligned} 
  &\mathcal{m}^{\frac{1}{p-1}}\le  \Big(1-(1-\mbox{RHS})\Big)^{\frac{1}{p-1}}\le 1-\frac{1}{p-1}(1-\mbox{RHS})\\
    &=1+\frac{1}{p-1}\frac{1}{a_\varepsilon^{\frac{p}{p-1}}}\Big[\frac{\theta+1}{\mu_{\alpha,\theta}}\ln \frac{\delta}{L}-A_0-\frac{1}{\mu_{\alpha,\theta}}\ln \frac{\omega_\theta}{\theta+1}+\frac{p-1}{\mu_{\alpha,\theta}}[\Psi(p)+\gamma]
+o_{\delta}(1)+o_L(1)+o_\varepsilon(1)\Big]\\
&=1+\Phi_{\varepsilon}(\delta, L),
    \end{aligned}
\end{equation}
where 
\begin{equation}\label{Phi-definition}
    \Phi_{\varepsilon}(\delta,L)=\frac{1}{p-1}\frac{1}{a_\varepsilon^{\frac{p}{p-1}}}\Big[\frac{\theta+1}{\mu_{\alpha,\theta}}\ln \frac{\delta}{L}-A_0-\frac{1}{\mu_{\alpha,\theta}}\ln \frac{\omega_\theta}{\theta+1}+\frac{p-1}{\mu_{\alpha,\theta}}[\Psi(p)+\gamma]
+o_{\delta}(1)+o_L(1)+o_\varepsilon(1)\Big].
\end{equation}
By \eqref{phipsi} and Lemma~\ref{limpsi}, we have 
$$
u_\varepsilon(r_\varepsilon L)=a_\varepsilon+\frac{1}{a_\varepsilon^{\frac{1}{p-1}}}\Big[-\frac{\theta+1}{\mu_{\alpha,\theta}}\ln L-\frac{1}{\mu_{\alpha,\theta}}\ln \frac{\omega_\theta}{\theta+1}+o_\varepsilon(1)+o_L(1)\Big].
$$
From Lemma~\ref{lemag}, we can write
\begin{equation*}
\begin{aligned}
    u_\varepsilon(\delta)=\frac{1}{a_\varepsilon^{\frac{1}{p-1}}}\Big[-\frac{\theta+1}{\mu_{\alpha,\theta}}\ln \delta+A_0+o_\delta(1)+o_\varepsilon(1)\Big].
\end{aligned}
\end{equation*}
By combining the two previous identities, we obtain
\begin{equation}\label{des2}
    \begin{aligned}
        u_\varepsilon(r_\varepsilon L)-u_\varepsilon(\delta)=&a_\varepsilon\Big\{1+\frac{1}{a_\varepsilon^{\frac{p}{p-1}}}\Big[\frac{\theta+1}{\mu_{\alpha,\theta}}\ln \frac{\delta}{L}- \frac{1}{\mu_{\alpha,\theta}}\ln\frac{\omega_\theta}{\theta+1}-A_0+o_\varepsilon(1)+o_\delta(1)+o_{L}(1)\Big]\Big\}.
    \end{aligned}
\end{equation}
Note  that 
$$
-1 \leq \frac{1}{a_\varepsilon^{\frac{p}{p-1}}}\Big[\frac{\theta+1}{\mu_{\alpha,\theta}}\ln \frac{\delta}{L}- \frac{1}{\mu_{\alpha,\theta}}\ln\frac{\omega_\theta}{\theta+1}-A_0+o_\varepsilon(1)+o_\delta(1)+o_{L}(1)\Big] \leq 0
$$
for  sufficiently small $\varepsilon, \delta$ and sufficiently large $L$. Thus, by  the Bernoulli's inequality
 $(1+t)^b\ge 1+bt$, for $t\ge -1$, if $b\ge 1$ and
using \eqref{des2}, we get
\begin{equation}\label{descima}
    \begin{aligned}
        (u_\varepsilon(r_\varepsilon L)-u_\varepsilon( \delta))^\frac{p}{p-1} &\geq a_\varepsilon^{\frac{p}{p-1}}+\frac{p}{p-1}\Big[\frac{\theta+1}{\mu_{\alpha,\theta}}\ln \frac{\delta}{L}- \frac{1}{\mu_{\alpha,\theta}}\ln\frac{\omega_\theta}{\theta+1}-A_0
        +o_\varepsilon(1)+o_\delta(1)+o_L(1)\Big].
    \end{aligned}
\end{equation}
Hence
\begin{equation}\label{des22}
    \begin{aligned}
         \mathcal{m}^{\frac{1}{p-1}}&\ge \mu_{\alpha,\theta} \frac{a_\varepsilon^{\frac{p}{p-1}}+\frac{p}{p-1}\Big[\frac{\theta+1}{\mu_{\alpha,\theta}}\ln \frac{\delta}{L}- \frac{1}{\mu_{\alpha,\theta}}\ln\frac{\omega_\theta}{\theta+1}-A_0
        +o_\varepsilon(1)+o_\delta(1)+o_L(1)\Big]}{(\theta+1)\ln\frac{\delta}{L}-(\theta+1)\ln r_{\varepsilon}}\\
       &=\frac{\mu_{\alpha,\theta}a_\varepsilon^{\frac{p}{p-1}}+\frac{p}{p-1}\Big[(\theta+1)\ln \frac{\delta}{L}-\ln\frac{\omega_\theta}{\theta+1}-\mu_{\alpha,\theta}A_0
        +o_\varepsilon(1)+o_\delta(1)+o_L(1)\Big]}{(\theta+1)\ln\frac{\delta}{L}-(\theta+1)\ln r_{\varepsilon}} .
    \end{aligned}
\end{equation}
Next, we will combine estimates \eqref{desbaixo} and \eqref{des22} to derive 
\eqref{teodesig} via Lemma~\ref{desig}.  To this end, we focus on the value of $(\theta+1)\ln r_{\varepsilon}$ in \eqref{des22}.
By definition of $r_\varepsilon$ in \eqref{rE}, we have
\begin{equation*}
 -(\theta+1)\ln r_\varepsilon= \mu_\varepsilon a_\varepsilon^{\frac{p}{p-1}}-\ln \frac{\lambda_\varepsilon}{a_\varepsilon^{\frac{p}{p-1}}}. 
\end{equation*}
Then, by combining \eqref{desbaixo} with \eqref{des22} we obtain
\begin{equation}
    \begin{aligned}
     \mu_{\alpha,\theta}a_\varepsilon^{\frac{p}{p-1}}+&\frac{p}{p-1}\Big[(\theta+1)\ln \frac{\delta}{L}-\ln\frac{\omega_\theta}{\theta+1}-\mu_{\alpha,\theta}A_0
        +o_\varepsilon(1)+o_\delta(1)+o_L(1)\Big] \\  &\le \Big((\theta+1)\ln\frac{\delta}{L}+\mu_\varepsilon a_\varepsilon^{\frac{p}{p-1}}-\ln \frac{\lambda_\varepsilon}{a_\varepsilon^{\frac{p}{p-1}}}\Big)\mathcal{m}^{\frac{1}{p-1}}\\
        & \le \Big((\theta+1)\ln\frac{\delta}{L}+\mu_\varepsilon a_\varepsilon^{\frac{p}{p-1}}-\ln \frac{\lambda_\varepsilon}{a_\varepsilon^{\frac{p}{p-1}}}\Big)(1+\Phi_{\varepsilon}(\delta, L)).
        \end{aligned}
\end{equation}
From this,
\begin{equation}\label{Phi-passo1}
\begin{aligned}
 \Big(1+\Phi_\varepsilon(\delta, L)\Big)\ln\frac{\lambda_\varepsilon}{a_\varepsilon^\frac{p}{p-1}}&\leq (\theta+1)\ln\frac{\delta}{L}+\mu_\varepsilon a_\varepsilon^{\frac{p}{p-1}}+\Phi_{\varepsilon}(\delta, L)(\theta+1)\ln\frac{\delta}{L}+\mu_\varepsilon a_\varepsilon^{\frac{p}{p-1}}\Phi_{\varepsilon}(\delta, L)\\
 &- \mu_{\alpha,\theta}a_\varepsilon^{\frac{p}{p-1}}-\frac{p}{p-1}\Big[(\theta+1)\ln \frac{\delta}{L}-\ln\frac{\omega_\theta}{\theta+1}-\mu_{\alpha,\theta}A_0
        +o_\varepsilon(1)+o_\delta(1)+o_L(1)\Big].
\end{aligned}
\end{equation}
Recalling that $\mu_{\varepsilon}=\mu_{\alpha,\theta}-\epsilon$ we can write
\begin{equation}
	\begin{aligned}\label{Phi-simples}
		 \mu_{\varepsilon}a_\varepsilon^{\frac{p}{p-1}}\Phi_{\varepsilon}(\delta, L)&=\mu_{\alpha,\theta}a_\varepsilon^{\frac{p}{p-1}}\Phi_{\varepsilon}(\delta, L)-\varepsilon a_\varepsilon^{\frac{p}{p-1}}\Phi_{\varepsilon}(\delta, L)\\
		 &=\frac{\theta+1}{p-1}\ln \frac{\delta}{L}-\frac{\mu_{\alpha,\theta}}{p-1}A_0-\frac{1}{p-1}\ln \frac{\omega_\theta}{\theta+1}+[\Psi(p)+\gamma]
		+o_{\delta}(1)+o_L(1)+o_\varepsilon(1)\\
		&- \frac{\varepsilon}{p-1}\Big[\frac{\theta+1}{\mu_{\alpha,\theta}}\ln \frac{\delta}{L}-A_0-\frac{1}{\mu_{\alpha,\theta}}\ln \frac{\omega_\theta}{\theta+1}+\frac{p-1}{\mu_{\alpha,\theta}}[\Psi(p)+\gamma]
		+o_{\delta}(1)+o_L(1)+o_\varepsilon(1)\Big].
	\end{aligned}
\end{equation}
Plugging \eqref{Phi-simples} into \eqref{Phi-passo1}, we derive
\begin{equation}\label{des10}
	\begin{aligned}
		\Big(1+\Phi_\varepsilon(\delta,L)\Big)\ln\frac{\lambda_\varepsilon}{a_\varepsilon^\frac{p}{p-1}} &\leq - \varepsilon a_\varepsilon^{\frac{p}{p-1}}+ \Big((\theta+1)\Phi_{\varepsilon}(\delta, L)-\frac{(\theta+1)\varepsilon}{(p-1)\mu_{\alpha,\theta}}\Big)\ln \frac{\delta}{L}\\
		&+\Big(1+\frac{\varepsilon}{(p-1)\mu_{\alpha,\theta}}\Big)\ln \frac{\omega_\theta}{\theta+1}+\Big(\mu_{\alpha,\theta}+\frac{\varepsilon}{p-1}\Big)A_0\\
		&+[\Psi(p)+\gamma]+o_\varepsilon(1)+o_L(1)+o_\delta(1)\\
	&	\le  \Big((\theta+1)\Phi_{\varepsilon}(\delta, L)-\frac{(\theta+1)\varepsilon}{(p-1)\mu_{\alpha,\theta}}\Big)\ln \frac{\delta}{L}\\
	&+\Big(1+\frac{\varepsilon}{(p-1)\mu_{\alpha,\theta}}\Big)\ln \frac{\omega_\theta}{\theta+1}+\Big(\mu_{\alpha,\theta}+\frac{\varepsilon}{p-1}\Big)A_0\\
	&+[\Psi(p)+\gamma]+o_\varepsilon(1)+o_L(1)+o_\delta(1).
	\end{aligned}
\end{equation}
From \eqref{lambda-from below} and \eqref{Phi-definition}, we conclude that 
$$
\Phi_\varepsilon(\delta,L)\to 0\;\;\mbox{and}\;\;\Phi_\varepsilon(\delta,L)\ln\frac{\lambda_\varepsilon}{a_\varepsilon^\frac{p}{p-1}} \to 0, \;\;\mbox{as}\;\; \varepsilon\to 0,
$$
for arbitrarily fixed $\delta, L>0$.  Hence, \eqref{des10} yields
\begin{equation*}
    \ln \frac{\lambda_\varepsilon}{a_\varepsilon^{\frac{p}{p-1}}} \leq \ln \frac{\omega_\theta}{\theta+1}+\mu_{\alpha,\theta}A_0+\Psi(p)+\gamma+o_L(1)+o_\varepsilon(1)+o_\delta(1) 
\end{equation*}
which ensures
$$
\lim_{\varepsilon \to 0} \frac{\lambda_\varepsilon}{a_\varepsilon^\frac{p}{p-1}}\leq \frac{\omega_\theta}{\theta+1}e^{\mu_{\alpha,\theta}A_0+\Psi(p)+\gamma}=e^{\mu_{\alpha,\theta}A_0+\Psi(p)+\gamma}|B_1|_{\theta}.
$$
So, Lemma~\ref{desig}  implies \eqref{teodesig}.
\end{proof}
\subsection{Test-function computations}
\begin{lemma} \label{Lemma-TF}There exists a family $(v_\varepsilon) \subset X^{1,p}_R(\alpha,\theta)$ such that
    \begin{equation}\label{contrad}
        H_\nu(v_\varepsilon) = 1 \quad \text{and} \quad \int_0^R e^{\mu_{\alpha,\theta}|v_\varepsilon|^{\frac{p}{p-1}}} \, \mathrm{d} \lambda_\theta > |B_R|_{\theta}+e^{\mu_{\alpha,\theta}A_0+\Psi(p)+\gamma}|B_1|_\theta,
    \end{equation}
    for all $\varepsilon > 0$ sufficiently small.
\end{lemma}
\begin{proof}
Let $ g $ and $ z $ be as defined in Lemma~\ref{lemag}-$(iii)$. 
For $ \varepsilon > 0 $, set $ L := -\ln \varepsilon$.  Then, let
 $(v_\varepsilon),\; \varepsilon>0$ be defined by
\begin{equation}\label{famitest}
    v_\varepsilon(r) =
    \begin{cases}
        c +\frac{1}{c^{\frac{1}{p-1}}} \left[- \frac{p-1}{\mu_{\alpha,\theta}} \ln \Big( 1 + c_0 \left(\frac{r}{\varepsilon} \right)^{\frac{\theta+1}{p-1}} \Big) + \Lambda_\varepsilon \right], & 0 \leq r \leq L\varepsilon \\
        \frac{1}{c^{\frac{1}{p-1}}}(g - \varphi z), & L\varepsilon \leq r \leq 2 L\varepsilon \\
        \frac{1}{c^{\frac{1}{p-1}}} g, & 2 L \varepsilon\le r \le R,
    \end{cases}
\end{equation}
where $c$ and $\Lambda_\varepsilon$ will be chosen later so that 
$H_\nu(v_\varepsilon) = 1$ and $v_\varepsilon \in X^{1,p}_R(\alpha,\theta)$.  Here, 
 $\varphi\in C^1[0,R] $ is a smooth function such that  $0\le \varphi\le 1$, 
$\varphi \equiv 1$ on $[0, L\varepsilon]$, 
$\varphi \equiv 0$ on $[2L\varepsilon,R]$ 
and $|\varphi^\prime| = O\left(\frac{1}{L\varepsilon}\right)$. To ensure that $v_\varepsilon \in X^{1,p}_R(\alpha,\theta)$, 
we choose $\Lambda_\varepsilon$ and $c$ so that 
$v_\varepsilon$ is continuous at $r = L\varepsilon$. 
Consequently, $\Lambda_\varepsilon$ and $c$ must satisfy
\begin{equation*}
 c + \frac{1}{c^{\frac{1}{p-1}}} \left[- \frac{p-1 }{\mu_{\alpha,\theta}} \ln \Big( 1 + c_0 L^{\frac{\theta+1}{p-1}} \Big) + \Lambda_\varepsilon \right]=\frac{g(L\varepsilon) -  z(L\varepsilon)}{c^\frac{1}{p-1}},
\end{equation*}
or equivalently,
\begin{equation}\label{Defidelambda}
\Lambda_\varepsilon=-c^{\frac{p}{p-1}}+\frac{p-1}{\mu_{\alpha,\theta}} \ln \Big( 1 + c_0 L^{\frac{\theta+1}{p-1}} \Big)-\frac{\theta+1}{\mu_{\alpha,\theta}}\ln (L\varepsilon)+A_0.
\end{equation}
Next, we compute expressions for $\|v'_\varepsilon\|_{L^{p}_\alpha}^{p}$ and $\|v_\varepsilon\|_{L^{p}_\theta}^{p}$. Since $z\in C^1[0,R]$  and $ z^{\prime}(0) = z(0)=0$, we have $(\varphi z)' = O(1)$ as $\varepsilon \to 0$. 
Moreover, $|g'| = O((L\varepsilon)^{-1})$ uniformly on 
$[L\varepsilon, 2L\varepsilon]$. 
Hence, for $\varepsilon$ sufficiently small,
\begin{equation*}
|v^\prime_\varepsilon|^{p} = \frac{1}{c^{\frac{p}{p-1}}}|g^\prime - (\varphi z)^\prime|^{p} 
= \frac{|g^\prime|^{p}}{c^{\frac{p}{p-1}}}|1 - O(L\varepsilon)|^{p} 
= \frac{1}{c^{\frac{p}{p-1}}}|g^\prime|^{p}(1 + O(L\varepsilon)),
\end{equation*}
uniformly on $[L\varepsilon,2L\varepsilon]$ as $\varepsilon \to 0$. Therefore, we obtain
\begin{equation*}
\begin{aligned}  \int_{L\varepsilon}^R|v_\varepsilon^\prime|^{p}\mathrm{d} \lambda_\alpha&=\int_{L\varepsilon}^{2L\varepsilon}|v_\varepsilon^\prime|^{p}\mathrm{d} \lambda_\alpha+\int_{2L\varepsilon}^R|v_\varepsilon^\prime|^{p}\mathrm{d} \lambda_\alpha\\
    &=
    \frac{1}{c^\frac{p}{p-1}}\Big[(1+O(L\varepsilon))\int_{L\varepsilon}^{2L\varepsilon}|g^\prime|^{p}\mathrm{d} \lambda_\alpha+\int_{2L\varepsilon}^R|g^\prime|^{p}\mathrm{d} \lambda_\alpha\Big]\\
    &=
    \frac{1}{c^{\frac{p}{p-1}}}\Big[\int_{L\varepsilon}^R|g^\prime|^{p}\mathrm{d}\lambda_\alpha + O(L\varepsilon)\int_{L\varepsilon}^{2L\varepsilon} |g^\prime|^{p}\mathrm{d} \lambda_\alpha\Big].
\end{aligned}
\end{equation*}
Note that $|g'| = O((L\varepsilon)^{-1})$ uniformly on 
$[L\varepsilon, 2L\varepsilon]$ implies
\begin{equation*}
O(L\varepsilon)\int_{L\varepsilon}^{2L\varepsilon}|g^\prime|^{p}\mathrm{d} \lambda_\alpha=O\Big(\frac{1}{(L\varepsilon)^{p-1}}\Big)\int_{L\varepsilon}^{2L\varepsilon} \mathrm{d}\lambda_{\alpha}=O(L\varepsilon).
\end{equation*}
Consequently,
\begin{equation*}
\int_{L\varepsilon}^R|v_\varepsilon^\prime|^{p}\mathrm{d} \lambda_\alpha= \frac{1}{c^{\frac{p}{p-1}}}\Big[\int_{L\varepsilon}^R|g^\prime|^{p}\mathrm{d}\lambda_\alpha + O(L\varepsilon)\Big].
\end{equation*}
To compute the above integral involving $g^{\prime}$, 
we choose the test function $z_\varepsilon(r)=\min\{g(r),g(L\varepsilon)\}$ in \eqref{eqG}. 
Then

\begin{equation}\label{g-fracaz}
    \int_0^R|g^\prime|^{p-2}g^\prime z_\varepsilon^\prime \mathrm{d}\lambda_\alpha=\delta_0(z_\varepsilon)+\nu \int_0^R|g|^{p-2}g z_\varepsilon \mathrm{d}\lambda_\theta.
\end{equation}
Now, it follows directly from \eqref{Definig} that
	\begin{equation}\label{g-LqO}
		\int_{0}^r|g|^{q}\mathrm{d}\lambda_{\theta}=O(r^{\theta+1}|\ln r|^{q}),\;\; q\ge 1\;\; \mbox{as}\;\; r\to 0.
	\end{equation}
Therefore, recalling that $g$ is a decreasing function, \eqref{g-fracaz} and \eqref{g-LqO} imply
\begin{equation*}
    \begin{aligned}
        \int_{L\varepsilon }^R|g^{\prime}|^p \mathrm{d}\lambda_\alpha&=g(\varepsilon L)+\nu g(\varepsilon L)\int_0^{L\varepsilon }|g|^{p-1}\mathrm{d}\lambda{_\theta}+\nu \int_{L\varepsilon}^R|g|^p \mathrm{d}\lambda_\theta\\
   &=\nu \|g\|^{p}_{L^{p}_\theta}-\frac{\theta+1}{\mu_{\alpha,\theta}}\ln L\varepsilon+A_0+z(L\varepsilon)+\nu g(L\varepsilon)\int_0^{L\varepsilon}|g|^{p-1} \mathrm{d}\lambda_\theta -\nu \int_0^{L \varepsilon}|g|^p \mathrm{d}\lambda_\theta\\
   &=\nu \|g\|^{p}_{L^{p}_\theta}-\frac{\theta+1}{\mu_{\alpha,\theta}}\ln L\varepsilon+A_0+o(L\varepsilon)+O((L\varepsilon)^{\theta+1}|\ln (L\varepsilon)|^p).
    \end{aligned}
\end{equation*}
It follows that
\begin{equation}\label{firstpart}
	\begin{aligned}
	\int_{L_\varepsilon}^R|v^\prime_\varepsilon|^{p}\mathrm{d}\lambda_\alpha& =\frac{1}{c^{\frac{p}{p-1}}}\Big[\nu\|g\|^{p}_{L^{p}_\theta}-\frac{\theta+1}{\mu_{\alpha,\theta}}\ln L\varepsilon+A_0+O(L\varepsilon)+O((L\varepsilon)^{\theta+1}|\ln (L\varepsilon)|^p)\Big]\\
	&=\frac{1}{c^{\frac{p}{p-1}}}\Big[\nu\|g\|^{p}_{L^{p}_\theta}-\frac{\theta+1}{\mu_{\alpha,\theta}}\ln L\varepsilon+A_0+O(L\varepsilon|\ln (L\varepsilon)|)\Big].
	\end{aligned}	
\end{equation}
By Lemma~\ref{lemadaintegral}, a direct computation yields
\begin{equation}\label{seconfpart}
\begin{aligned}
\int_{0}^{L\varepsilon}|v^\prime_\varepsilon|^{p}\mathrm{d} \lambda_\alpha
&=\frac{p-1}{c^{\frac{p}{p-1}}\mu_{\alpha,\theta}}\int_0^{c_0 L^{\frac{\theta+1}{p-1}}} \frac{s^{p-1}}{(1+s)^p} \mathrm{d} s
\\
&=\frac{p-1}{c^{\frac{p}{p-1}}\mu_{\alpha,\theta}}\Big[\ln(1+
c_0L^{\frac{\theta+1}{p-1}})-\Psi(p)-\gamma+O(L^{-\frac{\theta+1}{p-1}})\Big].
\end{aligned}
\end{equation}
Thus, using \eqref{firstpart} and \eqref{seconfpart}, we obtain
\begin{equation}\label{normaulinha}
    \begin{aligned}
       \|v^\prime_\varepsilon\|_{L^{p}_\alpha}^{p} &= \frac{1}{c^{\frac{p}{p-1}}}\Big\{\nu \|g\|^{p}_{L^{p}_\alpha}-\frac{\theta+1}{\mu_{\alpha,\theta}}\ln L\varepsilon+A_0+\frac{p-1}{\mu_{\alpha,\theta}}\Big[\ln(1+c_0 L^{\frac{\theta+1}{p-1}})\\
 &-\Psi(p)-\gamma+O(L^{-\frac{\theta+1}{p-1}})\Big]+O(L\varepsilon|\ln (L\varepsilon)|)\Big\}.
    \end{aligned}
\end{equation}
Next, we compute $\|v_\varepsilon\|_{L^{p}_\theta}^{p}$. 
By  \eqref{Definig}, we have
\begin{equation*}
v_\varepsilon=\frac{g}{c^{\frac{1}{p-1}}}\Big[1+\frac{O(L\varepsilon)}{\ln(L\varepsilon)}\Big]
\end{equation*}
uniformly on $[L\varepsilon,2L\varepsilon]$. 
Consequently, \eqref{g-LqO} yields
\begin{equation}\label{ufirtsinteval}
    \begin{aligned}
        \int_{L\varepsilon}^R |v_\varepsilon|^p \mathrm{d} \lambda_\theta&=\frac{1}{c^{\frac{p}{p-1}}}\Big[\int_{2L\varepsilon}^R|g|^p \mathrm{d}\lambda_{\theta} +\int_{L\varepsilon}^{2L\varepsilon}|g|^p\mathrm{d}\lambda_{\theta}+\frac{O(L\varepsilon)}{|\ln(L\varepsilon)|}\int_{L\varepsilon}^{2L\varepsilon}|g|^p\mathrm{d}\lambda_{\theta}\Big]\\
        &=\frac{1}{c^{\frac{p}{p-1}}}\Big[\|g\|^{p}_{L^{p}_\theta} +O((L\varepsilon)^{\theta+1}|\ln (L\varepsilon)|^p)\Big].
    \end{aligned}
\end{equation}
Substituting the expression for $c^{\frac{p}{p-1}}+\Lambda_{\varepsilon}$ obtained in \eqref{Defidelambda} into \eqref{famitest}, for $0 \leq r \leq L\varepsilon$, we obtain
\begin{equation*}
	\begin{aligned}
		c^{\frac{1}{p-1}}v_\varepsilon&=\frac{p-1}{\mu_{\alpha,\theta}}
		\ln\!\left(
		\frac{1 + c_0 L^{\frac{\theta+1}{p-1}}}
		{1 + c_0 \left(\frac{r}{\varepsilon}\right)^{\frac{\theta+1}{p-1}}}
		\right)-\frac{\theta+1}{\mu_{\alpha,\theta}}\ln(L\varepsilon)+A_0
	\end{aligned}
\end{equation*}
and so
\begin{equation*}
	  \int_{0}^{L\varepsilon} |v_\varepsilon|^q\mathrm{d} \lambda_\theta=\frac{1}{c^{\frac{q}{p-1}}}\Big[O\Big((L\varepsilon)^{\theta+1}|\ln(L\varepsilon)|^{q}\Big)\Big],
\end{equation*}
for $\varepsilon > 0$ sufficiently small. 
Combining this estimate with \eqref{ufirtsinteval} we obtain
\begin{equation}\label{unorm}
    \|v_\varepsilon\|^p_{L^{p}_{\theta}}=\frac{1}{c^{\frac{p}{p-1}}}\Big[\|g\|^{p}_{L^{p}_\theta} +O((L\varepsilon)^{\theta+1}|\ln (L\varepsilon)|^p)\Big].
\end{equation}
From \eqref{normaulinha} and \eqref{unorm},   we have $ H^p _{\nu}(v_\varepsilon)=1$ if and only if 
\begin{equation}\nonumber
	\begin{aligned}
		 c^{\frac{p}{p-1}}
		&=-\frac{\theta+1}{\mu_{\alpha,\theta}}\ln L\varepsilon+A_0+\frac{p-1}{\mu_{\alpha,\theta}}\Big[\ln(1+c_0 L^{\frac{\theta+1}{p-1}})
		-\Psi(p)-\gamma\Big]\\
		&+O(L^{-\frac{\theta+1}{p-1}})+O(L\varepsilon|\ln (L\varepsilon)|)+O((L\varepsilon)^{\theta+1}|\ln (L\varepsilon)|^p)\\
		&=-\frac{\theta+1}{\mu_{\alpha,\theta}}\ln\varepsilon+A_0+\frac{1}{\mu_{\alpha,\theta}}\ln\frac{\omega_{\theta}}{\theta+1}-\frac{p-1}{\mu_{\alpha,\theta}}\big[
		\Psi(p)+\gamma\big]+\frac{p-1}{\mu_{\alpha,\theta}}\ln\Big(1+\frac{1}{c_0 L^{\frac{\theta+1}{p-1}}}\Big)\\
		&+O(L^{-\frac{\theta+1}{p-1}})+O(L\varepsilon|\ln (L\varepsilon)|)+O((L\varepsilon)^{\theta+1}|\ln (L\varepsilon)|^p).
		  \end{aligned}
\end{equation}
Recalling that $L=-\ln\varepsilon$ (or $\varepsilon=e^{-L}$),  our suitable choice of $c$ is such that
\begin{equation}\label{assintoticoc}
	\begin{aligned}
		c^{\frac{p}{p-1}}
		&=-\frac{\theta+1}{\mu_{\alpha,\theta}}\ln\varepsilon+A_0+\frac{1}{\mu_{\alpha,\theta}}\ln\frac{\omega_{\theta}}{\theta+1}-\frac{p-1}{\mu_{\alpha,\theta}}\big[
		\Psi(p)+\gamma\big]+O(L^{-\frac{\theta+1}{p-1}}).
	\end{aligned}
\end{equation}
It remains to estimate 
$
\int_0^R e^{\mu_{\alpha,\theta}| v_\varepsilon|^{\frac{p}{p-1}}} \mathrm{d}\lambda_\theta.
$
Firstly, for any $p \ge 2$, there exists a constant $d_p > 0$ such that
$e^{x} \ge 1 + d_p\, x^{p-1} ,$ for all $x\ge 0$. Hence, by \eqref{ufirtsinteval} and using the fact that 
$c^{\frac{p}{p-1}} = O(-\ln\varepsilon)= O(L)$ in view of \eqref{assintoticoc}, we obtain
\begin{equation}\label{integraldentro}
    \begin{aligned}
        \int_{L\varepsilon}^R e^{\mu_{\alpha,\theta}| v_\varepsilon|^{\frac{p}{p-1}}} \mathrm{d}\lambda_\theta  \geq&\int_{L\varepsilon}^R\Big(1+d_p\mu_{\alpha,\theta}^{p-1}|v_\varepsilon|^p\Big) \mathrm{d}\lambda_\theta\\
        =&|B_R|_\theta-|B_{L\varepsilon}|_\theta+d_p\mu_{\alpha,\theta}^{p-1}\frac{1}{c^{\frac{p}{p-1}}}\Big[\|g\|^{p}_{L^{p}_\theta} +O((L\varepsilon)^{\theta+1}|\ln (L\varepsilon)|^p)\Big]\\
        =&|B_R|_\theta+d_p\mu_{\alpha,\theta}^{p-1}\frac{1}{c^{\frac{p}{p-1}}}\Big[\|g\|^{p}_{L^{p}_\theta} +O((L\varepsilon)^{\theta+1}|\ln (L\varepsilon)|^p)\Big]+O((L\varepsilon)^{\theta+1})\\
        =&|B_R|_\theta+d_p\mu_{\alpha,\theta}^{p-1}\frac{1}{c^{\frac{p}{p-1}}}\|g\|^{p}_{L^{p}_\theta}+O(L^{-\frac{\theta+1}{p-1}}).
    \end{aligned}  
\end{equation}
Using the inequality $|1+t|^a \geq 1+at$ for all $t \in \mathbb{R}$ and $a>1$, we get
\begin{equation*}
\begin{aligned}
    |v_\varepsilon|^{\frac{p}{p-1}} &\geq c^{\frac{p}{p-1}}-\frac{p}{\mu_{\alpha,\theta}}\ln\Big(1+c_0\big(\frac{r}{\varepsilon}\big)^\frac{\theta+1}{p-1}\Big)+\frac{p}{p-1}\Lambda_\varepsilon\\
    &=-\frac{1}{p-1}c^{\frac{p}{p-1}}-\frac{p}{\mu_{\alpha,\theta}}\ln\Big(1+c_0\big(\frac{r}{\varepsilon}\big)^\frac{\theta+1}{p-1}\Big)+\frac{p}{p-1}(c^\frac{p}{p-1}+\Lambda_\varepsilon),
\end{aligned}
\end{equation*}
for $r\in [0, L\varepsilon]$.
Using \eqref{Defidelambda} and \eqref{assintoticoc}, we obtain
\begin{equation*}
\begin{aligned}
    &|v_\varepsilon|^{\frac{p}{p-1}} \geq
    -\frac{1}{p-1}c^{\frac{p}{p-1}}-\frac{p}{\mu_{\alpha,\theta}}\ln\Big(1+c_0\big(\frac{r}{\varepsilon}\big)^\frac{\theta+1}{p-1}\Big)\\
    &+\frac{p}{\mu_{\alpha,\theta}}\ln \Big( 1 + c_0 L^{\frac{\theta+1}{p-1}} \Big)-\frac{\theta+1}{\mu_{\alpha,\theta}}\frac{p}{p-1}\ln (L\varepsilon)+\frac{p}{p-1}A_0\\
    &=-\frac{\theta+1}{\mu_{\alpha,\theta}}\ln\varepsilon+\frac{1}{\mu_{\alpha,\theta}}\ln\frac{\omega_{\theta}}{\theta+1}+A_0+\frac{1}{\mu_{\alpha,\theta}}\big[
    \Psi(p)+\gamma\big]-\frac{p}{\mu_{\alpha,\theta}}\ln\Big(1+c_0\big(\frac{r}{\varepsilon}\big)^\frac{\theta+1}{p-1}\Big)+O(L^{-\frac{\theta+1}{p-1}}).
\end{aligned}
\end{equation*}
Integrating on $[0, L\varepsilon]$ and making change of variables $r=\varepsilon s$, we get
\begin{equation*}
    \begin{aligned}
        \int_0^{L\varepsilon} e^{\mu_{\alpha,\theta}|v_\varepsilon|^{\frac{p}{p-1}}}\mathrm{d}\lambda_\theta \geq \frac{\omega_\theta}{\theta+1} e^{\mu_{\alpha,\theta}A_0+\Psi(p)+\gamma+O(L^{-\frac{\theta+1}{p-1}})}\int_0^L\frac{1}{\Big(1+c_0s^\frac{\theta+1}{p-1}\Big)^p} \mathrm{d}\lambda_\theta.
    \end{aligned}
\end{equation*}
Arguing as in \eqref{integral-cabulosa}, we can write
\begin{equation*}
    \int_0^L\frac{1}{\Big(1+c_0s^\frac{\theta+1}{p-1}\Big)^p} \mathrm{d}\lambda_\theta=1-\int_L^{\infty}\frac{1}{\Big(1+c_0s^\frac{\theta+1}{p-1}\Big)^p} \mathrm{d}\lambda_\theta=1+O\Big(L^{-\frac{\theta+1}{p-1}}\Big).
\end{equation*}
Therefore
\begin{equation}\label{integralfora}
    \begin{aligned}
        \int_0^{L\varepsilon} e^{\mu_{\alpha,\theta}|v_\varepsilon|^{\frac{p}{p-1}}}\mathrm{d}\lambda_\theta &\geq \frac{\omega_\theta}{\theta+1} e^{\mu_{\alpha,\theta}A_0+\Psi(p)+\gamma+O(L^{-\frac{\theta+1}{p-1}})}\Big(1+O(L^{-\frac{\theta+1}{p-1}})\Big)\\
        &=\frac{\omega_\theta}{\theta+1} e^{\mu_{\alpha,\theta}A_0+\Psi(p)+\gamma}+O(L^{-\frac{\theta+1}{p-1}}).
    \end{aligned}
\end{equation}
Combining \eqref{integraldentro} and \eqref{integralfora}, we deduce
\begin{equation*}
\begin{aligned}
    \int_0^{R} e^{\mu_{\alpha,\theta}|v_\varepsilon|^{\frac{p}{p-1}}}\mathrm{d}\lambda_\theta &\geq|B_R|_\theta+d_p\mu_{\alpha,\theta}^{p-1}\frac{1}{c^{\frac{p}{p-1}}}\|g\|^{p}_{L^{p}_\theta}+O(L^{-\frac{\theta+1}{p-1}})\\
    &+\frac{\omega_\theta}{\theta+1} e^{\mu_{\alpha,\theta}A_0+\Psi(p)+\gamma}+O(L^{-\frac{\theta+1}{p-1}})\\
    &=|B_R|_\theta+e^{\mu_{\alpha,\theta}A_0+\Psi(p)+\gamma}|B_1|_{\theta} \\
    &+\frac{1}{c^{\frac{p}{p-1}}}\Big[d_p\mu_{\alpha,\theta}^{p-1}\|g\|^{p}_{L^{p}_\theta}+O\Big(c^{\frac{p}{p-1}}L^{-\frac{\theta+1}{p-1}}\Big)\Big].
\end{aligned}
\end{equation*}
Since $\theta\ge p$, $L = -\ln \varepsilon$ and $c^{\frac{p}{p-1}} = O(-\ln\varepsilon)$, we obtain $O\Big(c^{\frac{p}{p-1}}L^{-\frac{\theta+1}{p-1}}\Big)\to 0$ as $\varepsilon\to 0$. Thus,  
\begin{equation*}
\begin{aligned}
    \int_0^{R} e^{\mu_{\alpha,\theta}|v_\varepsilon|^{\frac{p}{p-1}}}\mathrm{d}\lambda_\theta>|B_R|_\theta+ e^{\mu_{\alpha,\theta}A_0+\Psi(p)+\gamma}|B_1|_\theta,
\end{aligned}
\end{equation*}
$\varepsilon > 0$ sufficiently small.
\end{proof}

\end{document}